\documentclass[a4paper,12pt,review]{elsarticle}
\usepackage[margin=2.5cm]{geometry}
\usepackage[final]{showkeys}
\usepackage{fancyhdr}
\usepackage{amsmath}
\usepackage{amsfonts,amssymb,amsbsy,mathcomp}
\usepackage{upgreek}
\usepackage{appendix}
\usepackage{color}
\usepackage{caption}
\usepackage{subcaption}
\usepackage{verbatim}
\usepackage{booktabs}

\usepackage[figuresright]{rotating}

\usepackage{tabularx}
\usepackage{float}
\newcommand{\R}{{\mathbb R}}   %  set of real numbers
\newcommand{\sign}[1]{{\text{sign}}#1}

\newtheorem{theorem}{Theorem}

\newtheorem{lemma}[theorem]{Lemma}

\newtheorem{remark}[theorem]{Remark}
\newenvironment{proof}[1][Proof]{\textbf{#1.} }{\hfill \raisebox{-0.1em}{$\Box$}\\[2ex]}

\newcommand{\bigO}{{\mathcal{O}}}

\newcommand{\Ppol}{\mathbb{P}}

\newcommand{\blue}{\color{blue}}

\newcommand{\doi}[1]{{\texttt{ #1}}}

\definecolor{jancolor}{RGB}{180,0,180}
\definecolor{piotrcolor}{RGB}{0,0,255}
\definecolor{boriscolor}{RGB}{255,0,0}

\newif\ifrevision
\revisiontrue
\ifrevision

\else

\fi

\journal{Chemical Engineering Research and Design}

\begin{document}
\begin{frontmatter}
%\title{{\bf Numerical simulations for two-dimensional reaction-diffusion problems with formation of multiple dead zones}}
\title{{\bf Geometry-induced formation and merging of multiple dead zones in diffusion-reaction problems}}

\author[ad5]{Jan Valdman\corref{mycorr}}
\cortext[mycorr]{Corresponding author}
\ead{jan.valdman@utia.cas.cz}
\address[ad5]{Czech Academy of Sciences, Institute of Information Theory and Automation,
Pod vod\'arenskou v\v{e}\v z\'\i\ 4, 182 00, Prague 8, Czechia}

\address[ad4]{School of Digital Sciences and Engineering, Nazarbayev University, 
53 Kabanbay Batyr Ave., Astana, 010000, Kazakhstan}
\author[ad4]{Boris Golman}

\author[ad1]{Piotr Skrzypacz}
\address[ad1]{School of Sciences and Humanities, Nazarbayev University, 
53 Kabanbay Batyr Ave., Astana, 010000, Kazakhstan}

%=========================================================================================================================
\begin{abstract}
This study focuses on dead-core solutions for a two-dimensional isothermal reaction-diffusion problem involving a single chemical reaction described by power-law kinetics within a bounded reactor domain. The steady-state nonlinear diffusion-reaction boundary value problem is solved numerically by applying a suitable time-stepping technique.
The lumped finite element technique with piecewise linear shape functions is utilized to discretize the spatial domain. 
Dead zones, defined as regions inside a reactor where no reaction occurs because the reactant has been completely consumed, present a major difficulty in the design and operation of chemical reactors.
The impact of the reaction exponent and the Thiele modulus on the concentration fields and on the size of such inactive regions is numerically investigated. The findings indicate that particular reactor shapes can lead to the formation of several separate dead zones.
\\[1.5ex]
\end{abstract}

%=========================================================================================================================
\begin{keyword} 
diffusion-reaction equation \sep power-law reaction kinetics \sep 
dead zone\sep 
lumped finite elements\sep
time-marching scheme\\[2ex]
%=========================================================================================================================
\MSC[2010] 
%76A02 %Foundations of fluid mechanics
%\sep
80A30 %Chemical kinetics [See also 76V05, 92C45, 92E20] 
\sep
35K57 % Reaction-diffusion equations.
\sep
80A32 %Chemically reacting flows [See also 92C45, 92E20]
\sep
%76R50: Diffusion (fluid mechanics).
%sep
65N30 %FEM
\end{keyword}
\end{frontmatter}
%==================================================================================
{\bf Declarations of interest: none}\\
%===========================================================================================================
% Intro
%==============================================================================
\section{Introduction}\label{sec1}
%===================================================================================================================
The power-law rate equation is widely adopted for industrial reactions. For the single irreversible reaction
\[
Reactant~A\rightarrow Products\,,
\] it takes the form
\[
r_A = k_A C_A^p\,,
\]
with $r_A$  being the reaction rate, $C_A$ the concentration of reactant $A$, $p$ the reaction exponent, and $k_A$ the rate constant. In this paper we assume isothermal conditions, meaning that the reaction rate constant does not vary with temperature. The reaction exponents $p$ can be positive or negative integers and fractions. Fractional-order power-law kinetics appear in a variety of industrial reactions. Among these are the synthesis of methanol from carbon dioxide and hydrogen \cite{kobl_2016}, the reforming of methane with carbon dioxide \cite{paksoy_2019}, and the wet air oxidation of sodium sulfide in petroleum refinery processes \cite{barge_2020}.
The occurrence of dead zones--areas within a reactor where reactions cease due to reactant depletion--poses a significant challenge in chemical reactor engineering.
A wide range of important catalytic reactions in industrial practice are characterized by power-law kinetics that involve fractional exponents \cite{romero_88,spence_1995,szukiewicz_2019}. When the reaction exponent is fractional, the resulting kinetics can drive reactant concentrations to zero within specific spatial regions, known as {\it dead zones} or {\it dead cores}. The term {\it dead zone} was originally coined by Temkin \cite{temkin_1975} and has since become widely established in the literature \cite{szukiewicz_2019,skrzypacz_2020,skrzypacz_2022}. Such inactive regions can significantly reduce reactor performance \cite{shi_2013} and have been documented in diverse systems, including the hydrogenation of propylene with commercial catalysts \cite{szukiewicz_2019}, electricity production in microbial fuel cells \cite{islam_2019}, and biochemical reactions taking place in catalytic particles containing immobilized enzymes \cite{pereira_2016}. Therefore, to enhance the performance of catalytic reactors, it is essential to predict and reduce dead zones in advance. Eliminating or shrinking these unproductive areas leads to a more even distribution of reactants, improved selectivity, and greater overall productivity.

Model problems that involve dead zones are numerically challenging for standard solvers because the nonlinear reaction term becomes non-differentiable when the concentration approaches zero. Existing solvers reported in the literature are also quite inefficient \cite{Kelley,chen_2001,aziz_1988,barrett_1991,valdes_2008,solsvik_2013}. Recent studies \cite{skrzypacz_membr_2021,skrzypacz_membr_2025}  introduced an efficient numerical approach based on a modified Crank-Nicolson scheme to solve a steady-state nonlinear problem with power-law kinetics and a fractional reaction exponent. In contrast, model problems whose solutions lack dead zones can be solved numerically using conventional iterative solvers or Taylor expansion methods \cite{skrzypacz_taylor_2024}.

%A substantial amount of research has been devoted to examining dead-core formation within catalyst slabs. However, the overwhelming majority of these investigations has concentrated on one-dimensional diffusion-reaction problems that incorporate power-law kinetics (see, for example, \cite{skrzypacz_2020, skrzypacz_2022, skrzypacz_sr_2022, skrzypacz_2023, szukiewicz_2020, Boris_nonisothermal}). In contrast, relatively little progress has been made in understanding how multiple dead-cores arise in two-dimensional configurations.

Previous studies have primarily concentrated on dead-core formation in one-dimensional catalyst pellets or on the computation of a single inactive region (see, for example, \cite{skrzypacz_2020, skrzypacz_2022, skrzypacz_sr_2022, skrzypacz_2023, szukiewicz_2020, Boris_nonisothermal}). In contrast, the present work investigates situations in which several disconnected inactive regions may coexist within a two-dimensional reactor. Particular attention is devoted to the influence of the reactor geometry on the topology of the dead-zone and on transitions between disconnected and connected dead-zone configurations. 
The primary contribution of this paper is the identification and numerical analysis of geometry-induced multiple dead zones. By enabling more efficient catalyst placement, this approach not only optimizes reactor performance but also yields substantial savings in costly materials. These insights are directly applicable to enhancing the design of ring-shaped and hole-cylinder pellets for fixed-bed reactors \cite{ring_shaped_2020,ring_shaped_2025}.

In the following, we assume that the constant concentration is prescribed on the reactor boundary. The lumped finite element method for spatial discretization combined with the appropriate time-marching scheme is employed to compute the concentration profiles and to show the formation of multiple dead zones for particular reactor geometries and process parameters like the Thiele modulus and reaction exponents. The objective of this paper is to develop the {\it implicit-explicit (IMEX)} time-marching scheme for dead-zone problems posed over two dimensional reactor domains of various geometries for isothermal reactions with power-law kinetics of fractional orders. In contrast to previous investigations of dead-core formation in diffusion-reaction systems with two-dimensional reactor geometries, our study proposes an IMEX pseudo-time-stepping scheme combined with a lumped finite element method (FEM) for the efficient numerical simulation of dead-core dynamics. Unlike our time-marching approach, FEM schemes using regularization of the non-Lipschitz reaction term \cite{aziz_1988,barrett_1991} yield nonlinear algebraic systems that 
can require a large number of iterations for numerical solution.

The structure of this paper is as follows. The mathematical model is introduced in Section~\ref{sec:model}. The numerical methodology based on conforming piecewise linear finite elements and an implicit–explicit (IMEX) time-marching scheme is described in Section~\ref{sec:numerical}. Numerical results for various reactor geometries and reaction parameters are presented and discussed in Section~\ref{sec:results}. The organization of the MATLAB\textsuperscript{\tiny\textregistered}  implementation and code availability are summarized in Section~\ref{sec:code}. Finally, conclusions are drawn in Section~\ref{sec:conclusions}.

\medskip
\noindent
%-----------------------------------------------------
% notations for spaces
%_{
{\it Notation:}~The following standard notation is adopted throughout this paper. 
For any domain \( G \subset \mathbb{R}^2 \), we let \( H^m(G) \) represent 
the space of all \( L^2(G) \)-functions possessing weak derivatives up 
to order \( m \in \mathbb{N}_0 \) that also lie in \( L^2(G) \). 
The subset of functions in \( H^1(G) \) that vanish on the boundary 
is denoted by \( H_0^1(G) \), and the inner product in \( L^2(G) \) 
is written as \( (\cdot, \cdot)_G \). 
The symbols \( \|\cdot\|_{m,G} \) and \( |\cdot|_{m,G} \) indicate 
the standard norm and seminorm on \( H^m(G) \), respectively. 
When \( G = \Omega \), the subscript \( G \) is dropped for simplicity.
%_}
%-----------------------------------------------------
%===================================================================================================================
\section{Mathematical model}\label{sec:model}
%===================================================================================================================
Consider a bounded planar domain $\widetilde{\Omega}\subset\mathbb{R}^2$ that represents a chemical reactor, with its boundary $\widetilde{\Gamma}:=\partial\widetilde{\Omega}$ representing the reactor walls.
 The reactant transport inside $\widetilde{\Omega}$ is governed by diffusion and consumption through a single isothermal and irreversible reaction $A\rightarrow Products$ following power-law kinetics. Let $V\subset \widetilde{\Omega}$ be an arbitrary control volume. The steady-state isothermal mass balance is given as
\begin{equation}\label{eq_volume_balance}
 \int\limits_{\partial V} \boldsymbol{N}_A\cdot\boldsymbol{n}
 \,\mathrm{d}S +\int\limits_V r_A\,\mathrm{d}V =0\,, 
 \end{equation}
 where $\boldsymbol{N}_A(C_A)$ denotes the molar flux of species $A$ with concentration $C_A(\tilde{x})$ and $r_A(C_A)$ is its disappearance rate.    The first integral represents the net outward flux of $A$, while the second represents the total consumption inside the control volume $V$. Using the Fick's law
 \begin{equation*}
 \boldsymbol{N}_A=-D_{{\rm eff}, A}\nabla C_A,
 \end{equation*}
 it follows from Eq.~\eqref{eq_volume_balance} and the divergence theorem that
 \begin{equation*}
-\int\limits_V \nabla\cdot (D_{{\rm eff},A} \nabla C _A)\,\mathrm{d}V+\int\limits_V r_A(C_A)\,\mathrm{d}V=0\,.
\end{equation*}
 Here, $D_{\text{eff},A}$ is the effective diffusion coefficient of species $A$ within the catalyst structure. Since $V\subset\widetilde{\Omega}$ is arbitrary, \begin{equation}\label{reactionvolumedim}
\begin{array}{rcll}
-D_{{\rm eff},A}\Delta C_A+r_A(C_A)&=&0 &\text{in}\quad \widetilde{\Omega}\,,\\
C_A &=&  C_{A,b} &\text{on}\quad \widetilde{\Gamma}\,,
\end{array}
\end{equation}
where $D_{\text{eff},A}$ is assumed to be constant in $\widetilde{\Omega}$, and $C_{A,b}$ denotes the fixed concentration imposed on the reactor walls $\widetilde{\Gamma}$.
Let $r_A(C_A)=k_A[C_A]_+^p$ represent the power-law kinetics characterized by the rate constant $k_A>0$. 
The boundary value problem \eqref{reactionvolumedim} reads in the dimensionless form as follows \cite{skrzypacz_2020,Boris_nonisothermal,aris_vol12,bandle84}  
\begin{equation}\label{reactionvolume}
\begin{array}{rcll}
-\Delta u+\varphi^2 f(u)&=&0 & \text{in}\quad \Omega\,,\\
u &=& 1 & \text{on}\quad \Gamma\,,
\end{array}
\end{equation}
where $u=\frac{C_A}{C_{A,b}}$ stands for the dimensionless concentration, the dimensionless spatial coordinate is scaled by the reference reactor length $L$, 
$f(u)=[u]_+^p$ is related to the chemical kinetics and $\varphi$ denotes the Thiele modulus defined as 
\begin{equation}\label{Thielemodulus}
\varphi=\left(\frac{L^2k_AC_{A,b}^{p-1}}{D_{\text{eff},A}} \right)^{1/2}\,.
\end{equation}  
In the present study, we focus on power-law chemical kinetics featuring a fractional reaction exponent:
\begin{equation}\label{eq_fu_def}
f(u)=
\begin{cases}
[u]_+^p\,,& \quad p\in (0,1)\,,\\
\sign{([u]_+)}\,, & \quad p=0\,, 
\end{cases}
\end{equation} 
where $[u]_+=\max\{u,0\}$ and $\sign(\cdot)$ denotes the signum function defined as $\sign(u)=\frac{u}{|u|}$ for $u\neq 0$ and $\sign(0)=0$. Notice that in the case of $p\in (0,1)$, the function $f(u)$ lacks the Lipschitz continuity at $u=0$. However, one can show that $f(u)$ is H\"older continuous, i.e., it holds true $|f(u_1)-f(u_2)|\le \,|u_1-u_2|^p$ for all $u_1, u_2\in\R$. 
%%%%%%%%%%%%%%%%%%%%%%%%%%%%%%%%%%%%%%%%%%%%%%%%%%%%%%%%%%%%%%%%%%%%%%%%%%%%%%%%%%%%%%%%%%%%%%%%%%%%%%%%%%%%%%%%%%

To investigate the existence and uniqueness of the solution to \eqref{reactionvolume}, we first introduce the concept of a weak solution. The function $u\in H^1(\Omega)$ with $u\arrowvert_\Gamma=1$ 
is said to be a weak solution to \eqref{reactionvolume} if $u$ satisfies
\begin{equation}\label{weak_one_comp}
\bigl(\nabla u, \nabla v\bigr)+\varphi^2\bigl(f(u),v\bigr)=0\quad\text{for all}\quad v\in H^1_0(\Omega)\,.
\end{equation}
The weak solution to the boundary value problem \eqref{reactionvolume} exists and is unique; the proof can be found in \cite{aziz_1988,bandle84}.  Furthermore, the solution of Eq.~\eqref{weak_one_comp} satisfies \cite{bandle84}
\begin{equation}\label{eq_max_principle}
0\le u<1\quad\text{in}~\Omega\,.
\end{equation}
Alternatively, the weak solution to \eqref{reactionvolume} can be interpreted as a unique minimizer of the convex energy functional 
\begin{equation}\label{deadcore_energy}
\begin{split}
E[v]=\int\limits_\Omega \left(\frac{1}{2}|\nabla v|^2+\frac{\varphi^2}{p+1}[v]_+^{p+1}\right)\,\mathrm{d}x
\end{split}
\end{equation}
in the space of functions from $H^1(\Omega)$ having trace $v\arrowvert_\Gamma=1$. In the case of $p\in [0,1)$, the strong adsorption 
of chemical species can be faster than its supply by the diffusion across the reactor walls $\Gamma$. 
This can lead to the total depletion of reactant in some regions, the so called dead zones (or dead cores) 
\[
\Omega_{dz}:=\{x\in \Omega:\; u(x)=0\}\subset \Omega\,.
\]
No reaction takes place in these regions, which results in the inefficient use of typically expensive catalytic material. 
Hence, determining where dead zones form is a critical consideration when designing chemical reactors. Note that dead-zone formation can also occur for general kinetic expressions, such as the Langmuir–Hinshelwood–Hougen–Watson (LHHW) form 
$f(u)=\frac{u^p}{(1+\lambda u)^m}$ where $p\in [0,1)$, $m>0$, and $\lambda>0$ \cite{piotr_critT_LH}. 

It has been proved in \cite{friedman84} that the dead zone is convex if $\Omega$ is convex. 
Since the boundary of the dead zone $\partial\Omega_{dz}$ is a-priori not known, the considered class of problems can be also interpreted as free boundary value problems. Some analytical results concerning the location of dead-core and the regularity of dead-core solution have been reported in \cite{bandle84,friedman84}. From the analytical point of view the most important of them are $u\in C^1(\Omega)$ and $u\in C^{2,\alpha}(\Omega_{dz})$ with the H\"older index $\alpha=\alpha(p)$, and the convexity of the dead-core if the domain $\Omega$ is convex.

Let us briefly illustrate the dead-zone phenomenon by the following diffusion-reaction problem
\begin{equation}\label{deadcore_ball}
\begin{split}
-\Delta u+\varphi^2 [u]_+^p&=0 \quad \text{in}\quad B_1\,,\qquad u=1 \quad \text{on}\quad \partial B_1\,,
\end{split}
\end{equation}
where $B_1=\{ (x_1,x_2)\in \R^2:\; \sqrt{x_1^2+x_2^2}<1\}$ denotes the unit disc. The solution to \eqref{deadcore_ball} is radially symmetric, i.e., $u(x,y)=v(r)$, 
where $r(x_1,x_2)=\sqrt{x_1^2+x_2^2}$. Consequently, the boundary value problem \eqref{deadcore_ball} can be reduced to the following two-point boundary value problem
\begin{equation}\label{deadcore_1d}
\begin{split}
-v_{rr}-\frac{v_r}{r}+\varphi^2 [v]_+^p&=0 \quad \text{in}\quad (0,1)\,,\qquad v_r(0)=0, \quad v(1)=1\,,
\end{split}
\end{equation}
whose solution possesses the dead zone only if  $\varphi\ge \varphi^*$ where
\begin{equation}\label{eq_phi_crit}
\varphi^*=\frac{2}{1-p}\,,
\end{equation}
cf. \cite{skrzypacz_2020}. Here, $\varphi^*$ is the critical Thiele modulus corresponding to the appearance of zero-length dead zone, $r_{dz} = 0$, at the center of the unit disc, i.e., $u(0) = 0$.
The explicit form of the dead-core solution to the boundary value problem \eqref{deadcore_1d} (and consequently to \eqref{deadcore_ball}) for the critical Thiele modulus defined in \eqref{eq_phi_crit} is given by \cite{skrzypacz_2020}
\[
u^*(r)=r^{\frac{2}{1-p}}\,.
\]
It can be deduced from the maximum principles that at the constant reaction order the size of the dead zone increases with the increasing values of the Thiele modulus. The monotonic behaviour of the dead-zone size with respect to the increasing $\varphi$ can be
proved analytically by the maximum principle, cf. \cite{stakgold_98}. In \cite{sperb_95} some bounds for the predicted location of dead zones have been derived analytically. Notice that the dead-core solution to \eqref{deadcore_ball} is no more analytic on $\partial\Omega_{dz}$. 
On the one hand the non-dead-core solution of problem \eqref{deadcore_ball}, i.e., for $\varphi< \varphi^*$, is smooth and can be obtained numerically without any advanced iteration methods but it cannot be expressed in a closed form. On the other hand, 
the existing iterative solvers for the stationary problem are for the case of dead-core solutions, i.e., 
for $\varphi> \varphi^*$, rather inefficient \cite{Kelley,chen_2001,aziz_1988,barrett_1991,lemos2014}
%ramirez2008}.

We use the time-marching method to compute the steady-state solution $u(x)$. For this purpose, we consider the unsteady equation governing the time-dependent solution $\tilde{u}(x,t)$:
\begin{equation}\label{false_Tr}
\frac{\partial \tilde{u}}{\partial t}-\Delta\tilde{u}+\varphi^2f(\tilde{u})=0
\end{equation}
subject to the boundary and the initial conditions 
\begin{equation}\label{bcic_one_comp}
\begin{split}
\tilde{u}(\cdot,t)&=1\quad\text{on}\quad \partial\Omega,\quad\text{for}\quad t>0\,,\qquad \tilde{u}(x,0)=1\quad\text{for}\quad x\in\Omega\,.  
\end{split}
\end{equation}
A standard result from \cite{Lady68} states that the solution to the quasilinear problem by Eqs.~\eqref{false_Tr}-\eqref{bcic_one_comp} exists and is unique in the class of bounded functions if $f$ is continuous and non-decreasing. Particularly, it can be shown that $\tilde{u}$ preserves positivity for all $t>0$. 

According to the following lemma, the solution of \eqref{false_Tr}-\eqref{bcic_one_comp} approaches the steady-state solution of \eqref{reactionvolume} in the $L^2$-norm when time tends to infinity.
\begin{lemma}\label{lemma_stab}
	Let $\tilde{u}(x,t)$ be a solution to the time-dependent problem \eqref{false_Tr}-\eqref{bcic_one_comp}, and let $u(x)$ be a solution to the boundary value problem \eqref{reactionvolume}. Then, it holds $\lim\limits_{t\to\infty} \tilde{u}(\cdot,t) = u$ in $L^2(\Omega)$. The convergence in the $L^2$-norm of $\tilde{u}(\cdot,t)$ towards the steady-state solution $u$ is exponentially fast for increasing $t>0$
	\begin{equation}\label{eq_ineq_expconv}
	\|\tilde{u}(\cdot,t)-u\|_0\le \|\tilde{u}(\cdot,0)-u\|_0\,e^{-\frac{1}{M_\Omega^2}\,t}\,,
	\end{equation}
where $M_\Omega>0$ is a positive constant depending on $\Omega$.
\end{lemma} 
\begin{proof}
The convergence of the transient solution towards the stationary solution follows from a standard energy argument. Testing the difference of Eqs.~\eqref{reactionvolume} and \eqref{false_Tr} by the error $\tilde{u}(\cdot,t)-u(\cdot)\in H^1_{0}(\Omega)$ yields	
    %Subtracting Eq.\eqref{reactionvolume} from Eq.\eqref{false_Tr}, using $\frac{\partial u}{\partial t}=0$, multiplying the resulting equation by the test function $\tilde{u}(\cdot,t)-u(\cdot)\in H^1_{0}(\Omega)$  and integrating by parts over $\Omega$ with respect to the spatial variable, yields	
    \begin{equation*}%\label{proof_lem2_integral}
	\begin{split}
	&\frac{1}{2}\frac{d}{dt}\|\tilde{u}(\cdot,t)-u\|_0^2+|\tilde{u}(\cdot,t)-u|_1^2+\varphi^2\Bigl( [\tilde{u}]_+^p(\cdot,t)-[u]_+^p, \tilde{u}(\cdot,t)-u\Bigr)=0\,.
	\end{split}
	\end{equation*}
	Notice that $(\alpha^p-\beta^p)(\alpha-\beta)\ge 0$ for all $\alpha,\beta\ge 0$ and $p\in (0,1]$. Then, we obtain
	\[
	\frac{d}{dt}\|\tilde{u}(\cdot,t)-u\|_0^2\le -2|\tilde{u}(\cdot,t)-u|_1^2\,.
	\]
Next, we employ the Poincar\'e-type inequality
	\[
	\|\tilde{u}(\cdot,t)-u\|_0\le M_\Omega|\tilde{u}(\cdot,t)-u|_1\,,
	\]
	where the positive constant $M_\Omega$ depends on the size of the reactor domain $\Omega$. The constant $M_\Omega$ is characterized as follows
    \[
    M_{\Omega} = \sup_{v \in H^1_0(\Omega),\; v \neq 0} \frac{\|v\|_{0}}{|v|_{1}}\,,
    \]
    or equivalently as
    \[
    M_{\Omega} = \frac{1}{\sqrt{\lambda_1}},
    \]
    where $\lambda_1$
  is the smallest eigenvalue of the Dirichlet Laplacian on $\Omega$  \cite{henrot}:
    \[
    -\Delta v = \lambda v \quad \text{in } \Omega, \qquad v = 0 \quad \text{on } \partial\Omega\,.
    \]
The first Dirichlet eigenvalue satisfies the Faber--Krahn
inequality
\[
\lambda_1(\Omega)
\geq
\frac{\pi}{|\Omega|}j_{0,1}^{\,2}\approx \frac{18.1684}{|\Omega|}\,,
\]
where \(|\Omega|\) denotes the volume of $\Omega$ and $j_{0,1}\approx 2.4048255577\ldots$ is the first positive root of the Bessel function of the first kind of order zero, \(J_0\). Note that equality holds if and only if \(\Omega\) is a ball
(up to sets of measure zero) \cite{henrot}.
    Then, the inequality of Poincar\'e-type implies that 
	\[
	\frac{d}{dt}\|\tilde{u}(\cdot,t)-u\|_0^2\le -\frac{2}{M_\Omega^2}\,\|\tilde{u}(\cdot,t)-u\|_0^2\,.
	\]
    From Gronwall's inequality we derive \eqref{eq_ineq_expconv}, which guarantees that 
    $\tilde{u}(\cdot,t)$ converges exponentially to $u$ in the $L^2$-norm as $t\to\infty$.
\end{proof}
An analogous lemma for the exponentially fast pointwise convergence has been established in \cite{Ricci}.
Notice that the statement of Lemma~\ref{lemma_stab} can be generalized to the case of monotonically increasing functions $f(u)$.

%===================================================================================================================
\section{Numerical method}\label{sec:numerical}

\subsection{Spatial FEM discretization}

For the finite element discretization of \eqref{weak_one_comp}, we consider a shape-regular family $\{\mathcal{T}_h\}$ of triangulations of $\Omega$. Examples of computational meshes are shown in Figure~\ref{fig:fourholes_mesh}. The meshes and all forthcoming numerical figures were generated using the Matlab\textsuperscript{\tiny\textregistered} PDE Toolbox package {\it pdetool}. 
The only exception is the unit-disk domain, for which we constructed a
partially rotationally structured mesh to better represent the expected
circular level sets and dead-core interface.

\begin{figure}[H]
    \centering
     \includegraphics[width=0.45\linewidth]{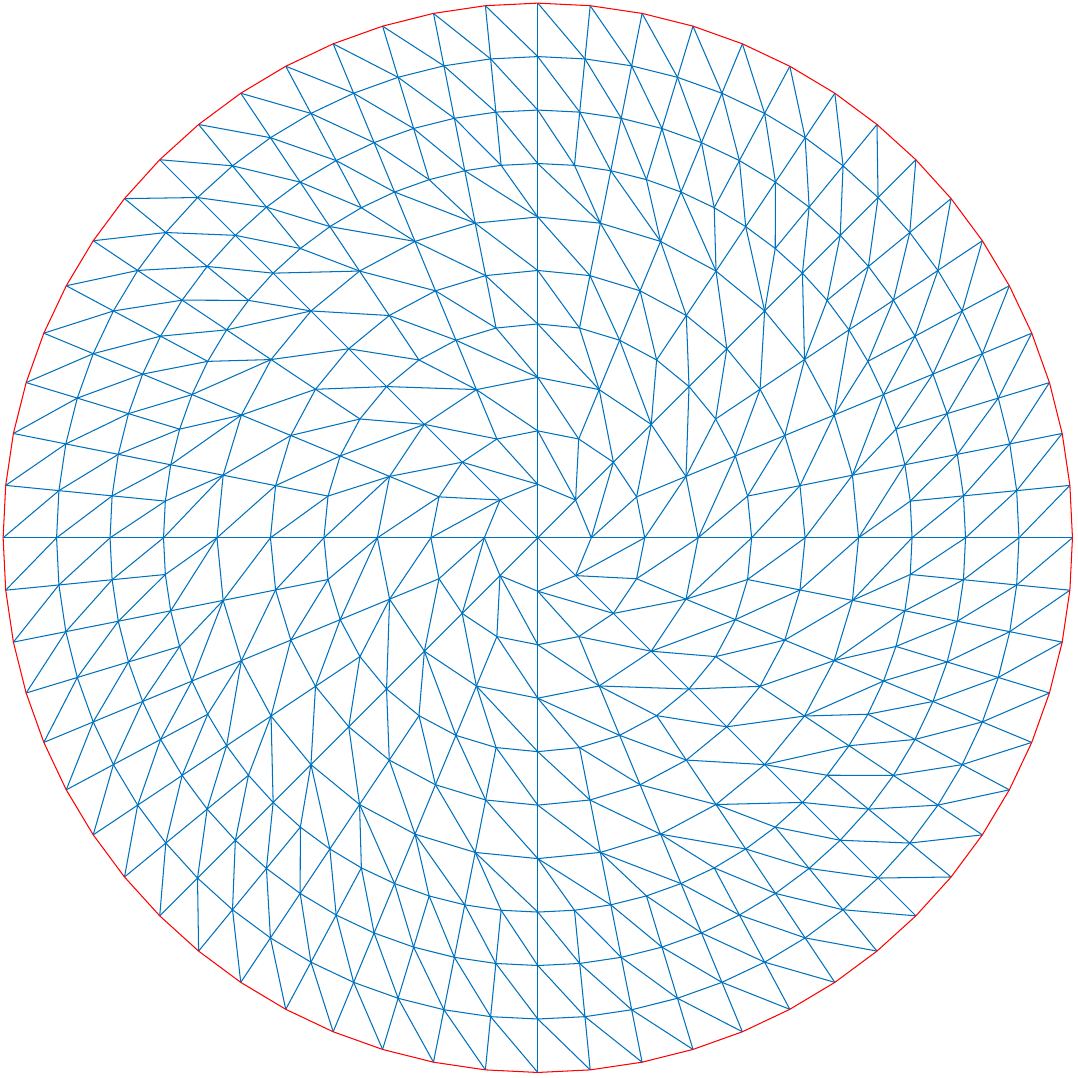}
     \hfill
    \includegraphics[width=0.45\linewidth]{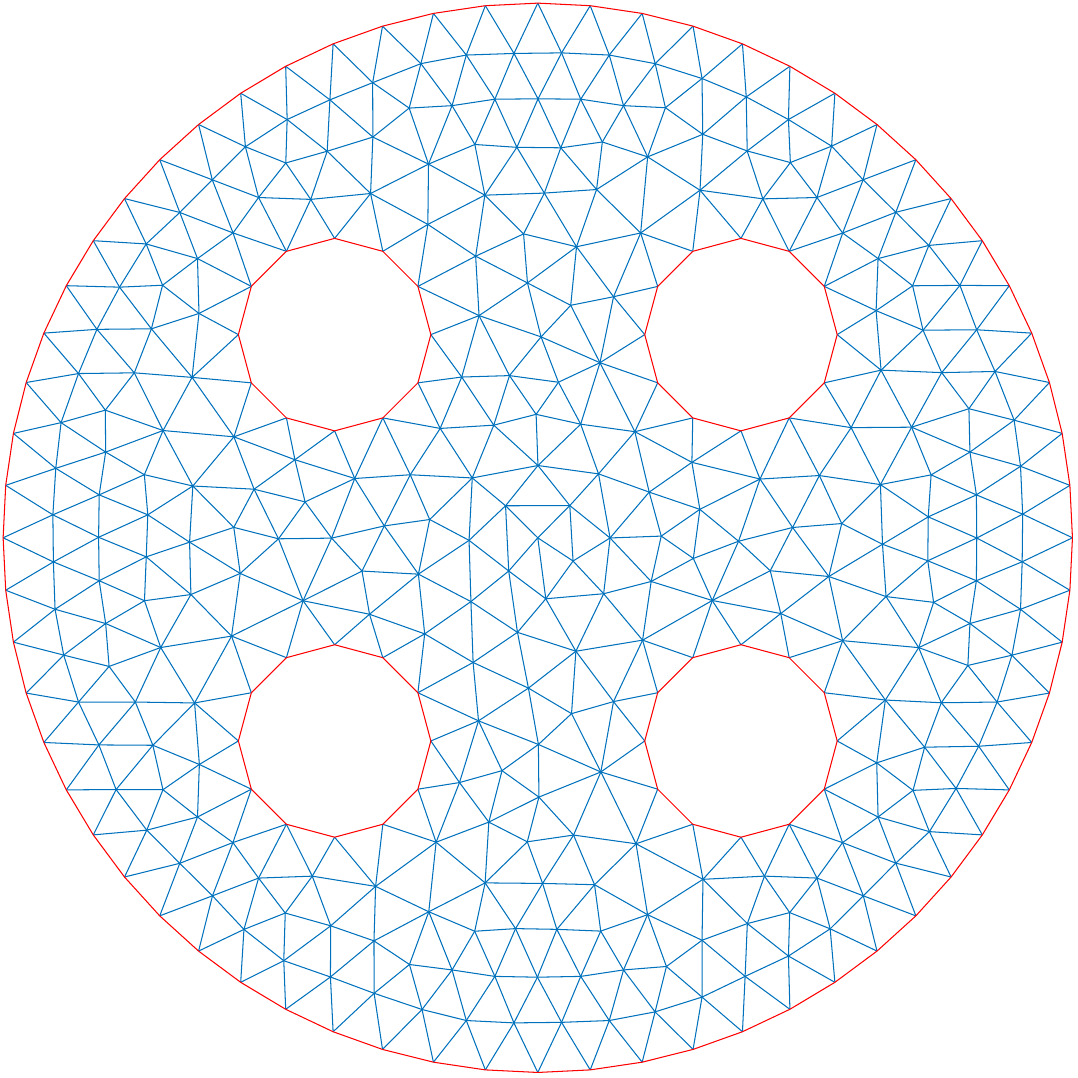}
   % \hfill
   % \includegraphics[width=0.45\linewidth]{disk_mesh.png}
    \caption{
Examples of triangular finite element meshes.
Left: custom partially rotationally structured mesh for the unit disk,
designed to better represent circular level sets.
Right: computational mesh for the multiply connected domain $\Omega$
with four circular holes.
} 
    %Left: multiply connected domain with four holes. Right: unit disk mesh aligned with the expected dead-core boundary.}
    \label{fig:fourholes_mesh}
\end{figure}

The diameter of an element $K\in\mathcal{T}_h$ is denoted by $h_K$, and the global mesh-size parameter is defined by
\[
h := \max_{K\in\mathcal{T}_h} h_K.
\]
 Let $\Ppol_1(K)$ denote the space of linear polynomials defined on the cell $K$. We define by 
\[
V_h=\{v_h\in C(\overline{\Omega})\,:\, v_h\arrowvert_K \in \Ppol_1(K)\quad\forall K\in {\mathcal{T}}_h\}
\] 
the space of piecewise linears with respect to $\{\mathcal{T}_h\}$, and we set $V_{h,0}=V_h\cap H^1_0(\Omega)$ for the discrete test space. Furthermore, we denote by $\varphi_i$ the nodal piecewise linear basis function associated with the $i$-th vertex in the decomposition $\{{\mathcal{T}}_h\}$. 

The finite element discretization of \eqref{weak_one_comp} 
reads as follows:\\[1.5ex] 
\hspace*{1cm} Find $u_h\in V_h$ with $u_h\arrowvert_\Gamma=1$ such that
\begin{equation}\label{Gal_stat}
(\nabla u_h, \nabla v_h) + \varphi^2\left< f(u_h), v_h\right>^h = 0\quad\text{for all}\quad v_h\in V_{h,0}\,,
\end{equation}
%The optimal error estimate $|u-u_h|_1\le C(u,\varphi,n)\,h$ has been derived in \cite{Barrett}.
where
\begin{equation}
  \left< f(u_h), v_h\right>^h=\sum\limits_{j=1}^{N_{VT}} \frac{|\Omega_j|}{3}f\bigl(u_h(a_j)\bigr) v_h(a_j)\,.
\end{equation}
Here, $N_{VT}$ denotes the total number of vertices in the decomposition $\{{\mathcal{T}}_h\}$, and
\[
\overline{\Omega}_i=\bigcup\limits_{K\in {\mathcal{T}}_h(a_i)} \overline{K}
\]
is a patch  that consists of all mesh cells belonging to the support of $\varphi_i$.
In Eq.~\eqref{Gal_stat}, the nonlinear term $\left< f(u_h), v_h\right>^h\approx \left( f(u_h), v_h\right)$ results from the mass lumping based on the following quadrature
\begin{equation*}
\int\limits_K w(x)\, dx \approx \frac{|K|}{3}\sum\limits_{\ell=1}^3 w(a_{K,\ell})\,,
\end{equation*}
where $\{a_{K,\ell}\}_{\ell=1}^3$ are the vertices of the triangular mesh cell $K\in\{\mathcal{T}_h\}$ with its area $|K|$.

\subsection{Time-marching approach}

The time-marching towards the steady-state solution is performed using the equidistant time points $ 0=t_0<t_1<\ldots <t_N=T$ where $T$ is chosen sufficiently large. We denote by $\tau=T/N$ the uniform time step size. Let $u^0_h\equiv 1$, and  let 
\[
u^k_h(x)=\sum\limits_{j=1}^{N_{VT}} u_j^k\varphi_j(x)\approx \tilde{u}(t_k,x)
\]
for $k\ge 1$ be the finite
element approximation from the previous time. Furthermore, we introduce the nodal vector with components associated with the interior vertices as
\[
\underline{u}^k\approx [u_1(t_k),\ldots, u_{N_F}(t_k)]^T
\]
where $N_F=N_{VT}-N_D$ and $N_D$ denotes the number of boundary vertices.

We propose the following implicit-explicit (IMEX) time-marching scheme for the computation of the nodal vector $\underline{u}^{k+1}\in\R^{N_F}$:
\begin{equation}\label{time_marching_scheme}
\left[
M^{FF}
+\tau A^{FF}
\right]
\underline{u}^{k+1}
=
M^{FF}\underline{u}^k
-
\tau M^{FF}\underline{f}(\underline{u}^k)
-
\tau A^{FD}\underline{u}^D
-
M^{FD}\underline{u}^D.
\end{equation}
Here, the vectors are defined by
\begin{equation*}
\begin{split}
\underline{f}(\underline{u})
&=
[f(u_1),\ldots,f(u_{N_F})]^T
\in\R^{N_F},
\\
\underline{u}^D
&=
[1,\ldots,1]^T
\in\R^{N_D},
\end{split}
\end{equation*}
while the stiffness and mass matrices are given by
\begin{equation*}
\begin{split}
A^{FF}_{ij}
&=
(\nabla\varphi_j,\nabla\varphi_i),
\qquad
M^{FF}_{ij}
=
\langle\varphi_j,\varphi_i\rangle^h,
\qquad
i,j=1,\ldots,N_F,
\\
A^{FD}_{ij}
&=
(\nabla\varphi_j,\nabla\varphi_i),
\qquad
M^{FD}_{ij}
=
\langle\varphi_j,\varphi_i\rangle^h,
\qquad
i=1,\ldots,N_F,
\quad
j=1,\ldots,N_D.
\end{split}
\end{equation*}
For practical computations, the matrix on the left-hand side remains constant throughout the iteration process, and the resulting linear systems are solved efficiently using a precomputed Cholesky factorization.
After every iteration step, we apply the positivity projection
\[
\underline{u}^{k+1}
\leftarrow
\max\{\underline{u}^{k+1},0\}.
\]
Additionally, to improve robustness and enforce monotone energy decay, the tentative iterate is damped according to
\[
\underline{u}^{k+1}
\leftarrow
\underline{u}^k
+
\beta
\left(
\underline{u}^{k+1}
-
\underline{u}^k
\right),
\]

where the damping parameter $\beta\in(0,1]$ is determined by a simple backtracking line-search procedure ensuring
\[
E_h(u_h^{k+1})
\le
E_h(u_h^k).
\]
The discrete energy functional used in the damping procedure is given by
\[
E_h(u_h)
=
\frac12\,\underline{u}_F^T A^{FF}\underline{u}_F
+
\frac{\varphi^2}{p+1}
\sum_{K\in\mathcal{T}_h}
\frac{|K|}{3}
\sum_{\ell=1}^3
\max\!\left(u_h(x_{K,\ell}),0\right)^{p+1},
\]
where $|K|$ denotes the area of the triangle $K$, and
$x_{K,\ell}$, $\ell=1,2,3$, are the quadrature points associated with the second-order quadrature rule on $K$.

Let $\Omega_h=\bigcup\limits_{K\in {\cal{T}}_h} \overline{K}$. The error bound for the piece-wise linear finite element solution of \eqref{Gal_stat} is given by \cite[Theorem 3.1]{barrett_1991}
\begin{equation}\label{fem_error}
|u-u_h|_{1,\Omega_h}\le Ch\,,
\end{equation}
provided that $h$ is sufficiently small and certain non-degeneracy properties of the solution $u$ hold.
%===================================================================================================================

\section{Numerical results}
\label{sec:results}

We compute dead-core solutions using the IMEX time-marching scheme
\eqref{time_marching_scheme} for several computational domains and a
sequence of nested uniform meshes with decreasing mesh size $h$. We
consider five computational meshes with
\[
h \in \{0.2,\ 0.1,\ 0.05,\ 0.025,\ 0.0125\},
\]
and use the corresponding time step
\[
\tau=h^2.
\]
For each mesh size, the number of time steps is chosen as
\[
N_{\mathrm{steps}}=\frac{40}{h},
\]
giving
\[
N_{\mathrm{steps}}\in
\{200,\ 400,\ 800,\ 1600,\ 3200\}.
\]
The number of time steps is chosen empirically so that the discrete
energy has stabilized and no visible evolution of the numerical
solution is observed.

\subsection{Benchmark problem on the unit disk}
\label{sec:benchmark_unit_disk}

The unit-disk problem provides a benchmark for assessing the accuracy
of the proposed finite element approximation in a setting where a
high-accuracy reference solution is available.
The dead zone is a single connected circular region bounded by one
circular interface, as shown in Figure~\ref{fig:disk_deadcore}.
This structure follows from the radial symmetry of the solution.

\begin{figure}[H]
\centering
\begin{minipage}[t]{0.47\linewidth}
    \centering
    \includegraphics[width=\linewidth]{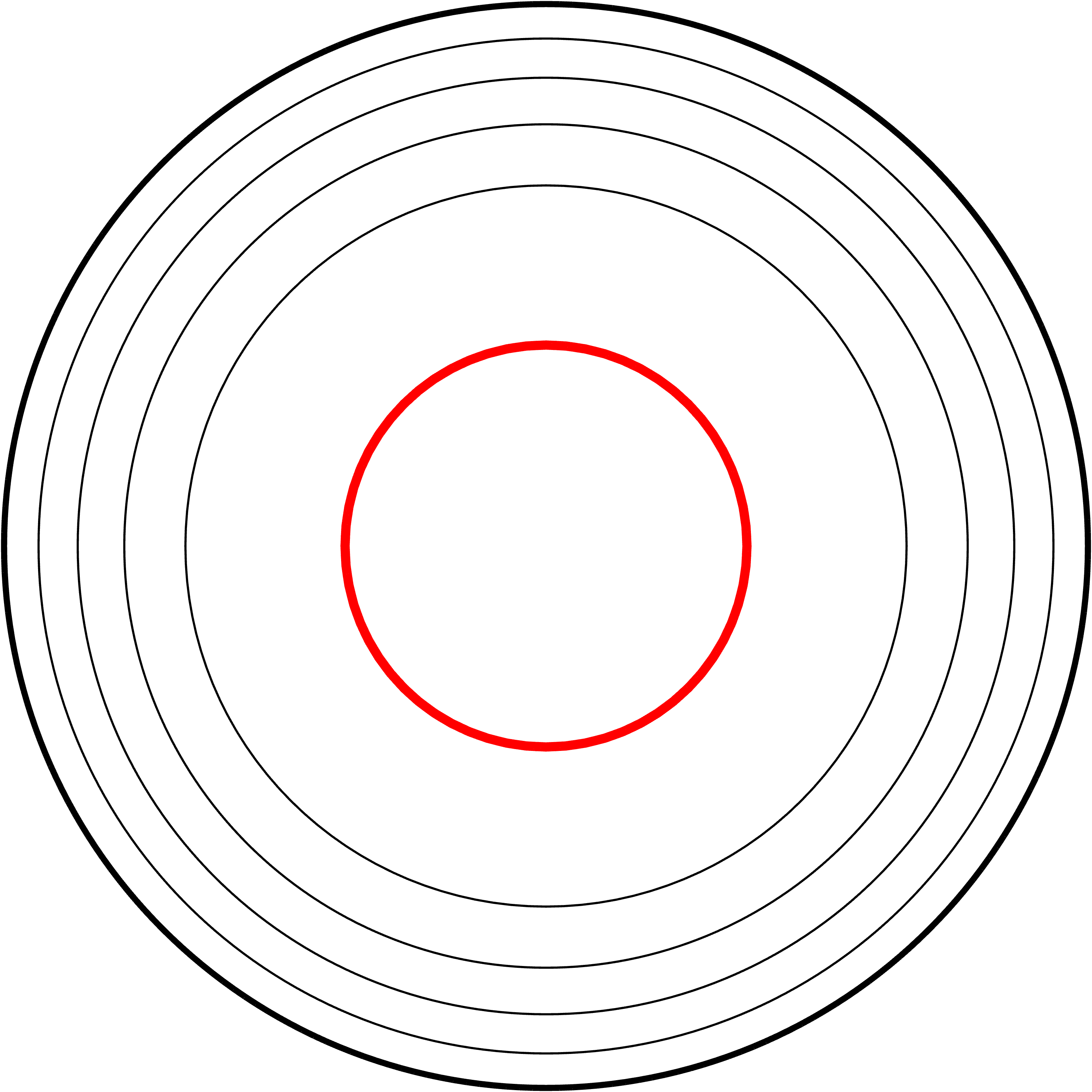}
\end{minipage}
\hfill
\begin{minipage}[t]{0.47\linewidth}
    \centering
    \includegraphics[width=\linewidth]{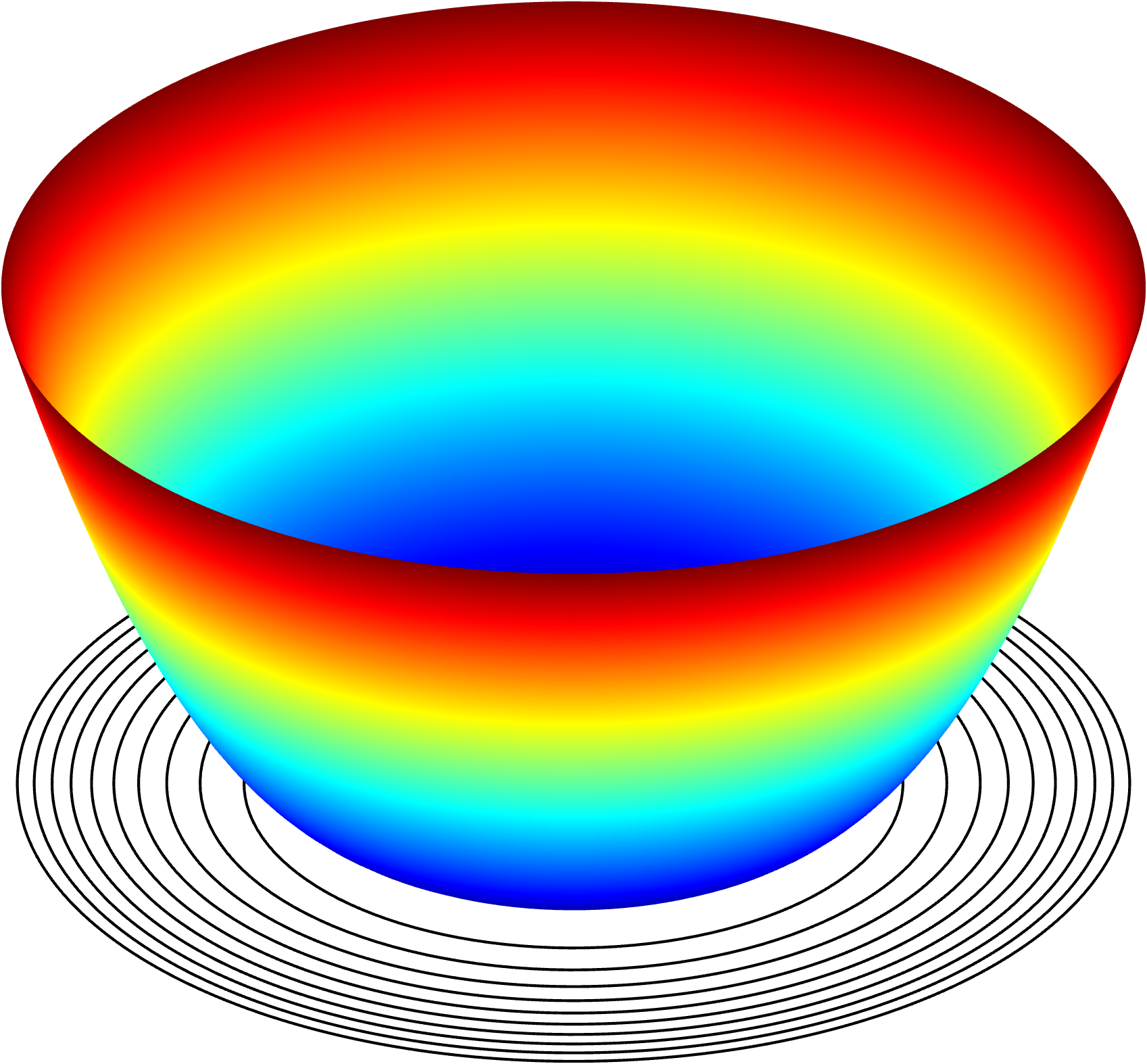}
\end{minipage}
\caption{Finite element approximation of the benchmark problem on the
unit disk for $\varphi=3$ and $p=0.1$. Left: contour plot of the
numerical solution with the reconstructed dead-core boundary using
$\alpha=10^{-5}$, shown in red. Right: three-dimensional visualization
of the numerical solution.}
\label{fig:disk_deadcore}
\end{figure}

Owing to this symmetry, the two-dimensional problem admits a
one-dimensional radial representation. This radial problem is used
to construct the high-accuracy reference solution and to evaluate
the numerical errors of the finite element approximation.

\subsubsection{High-accuracy radial reference solution}
\label{sec:radial_reference}

For the unit disk,
\(
\Omega=\{(x,y)\in\mathbb{R}^2:x^2+y^2<1\},
\)
the radially symmetric solution is written as
\[
u(x,y)=v(r),
\qquad
r=\sqrt{x^2+y^2}.
\]
The radial problem~\eqref{deadcore_1d} is characterized by an unknown
dead-core radius \(r_0\). The solution satisfies
\[
v(r)=0
\qquad\text{for}\qquad
0\le r\le r_0,
\]
while in the positive region \(r_0<r<1\),
\[
v''(r)+\frac{1}{r}v'(r)
=
\varphi^2[v(r)]^p.
\]
At the dead-core boundary and at the outer boundary,
\[
v(r_0)=0,
\qquad
v'(r_0)=0,
\qquad
v(1)=1.
\]

The unknown dead-core radius \(r_0\) was determined numerically by a
shooting procedure combined with bisection. For each trial value of
\(r_0\), the positive branch of the radial problem was integrated from
a point slightly outside the interface to \(r=1\), using the local
dead-core asymptotic behaviour to initialize the integration. The
resulting value of \(v(1)\) was compared with the prescribed boundary
value, and the interval containing the admissible value of \(r_0\)
was successively reduced.

After determining \(r_0\), the radial profile was reconstructed on a
fine one-dimensional grid, with \(v(r)=0\) inside the dead core. The
resulting profile was subsequently interpolated to the finite element
nodes to construct the reference solution on the two-dimensional
unit disk.

The reference energy was evaluated from the energy functional
\eqref{deadcore_energy}. By radial symmetry, the corresponding
two-dimensional integral reduces to
\[
E_{\mathrm{ref}}
=
2\pi\int_0^1
\left[
\frac{1}{2}\bigl(v'(r)\bigr)^2
+
\frac{\varphi^2}{p+1}[v(r)]_+^{p+1}
\right]r\,\mathrm{d}r.
\]

\begin{remark}
The benchmark problem considered throughout this section corresponds to
\[
\varphi=3,
\qquad
p=0.1.
\]
Its high-accuracy radial reference solution is shown in
Figure~\ref{fig:benchmark_radial_profile}. To assess the sensitivity
to the reaction exponent, additional values of $p$ are considered.
Here, $A_d/A_\Omega$ denotes the fraction of the reactor domain occupied
by the dead zone. The corresponding reference dead-core radii, dead-core
area fractions, and total energies are summarized in
Table~\ref{tab:p_sensitivity}. The first row corresponds to the
benchmark case used for the finite element validation.
\end{remark}

\begin{figure}[ht]
    \centering
     \includegraphics[width=0.82\textwidth]{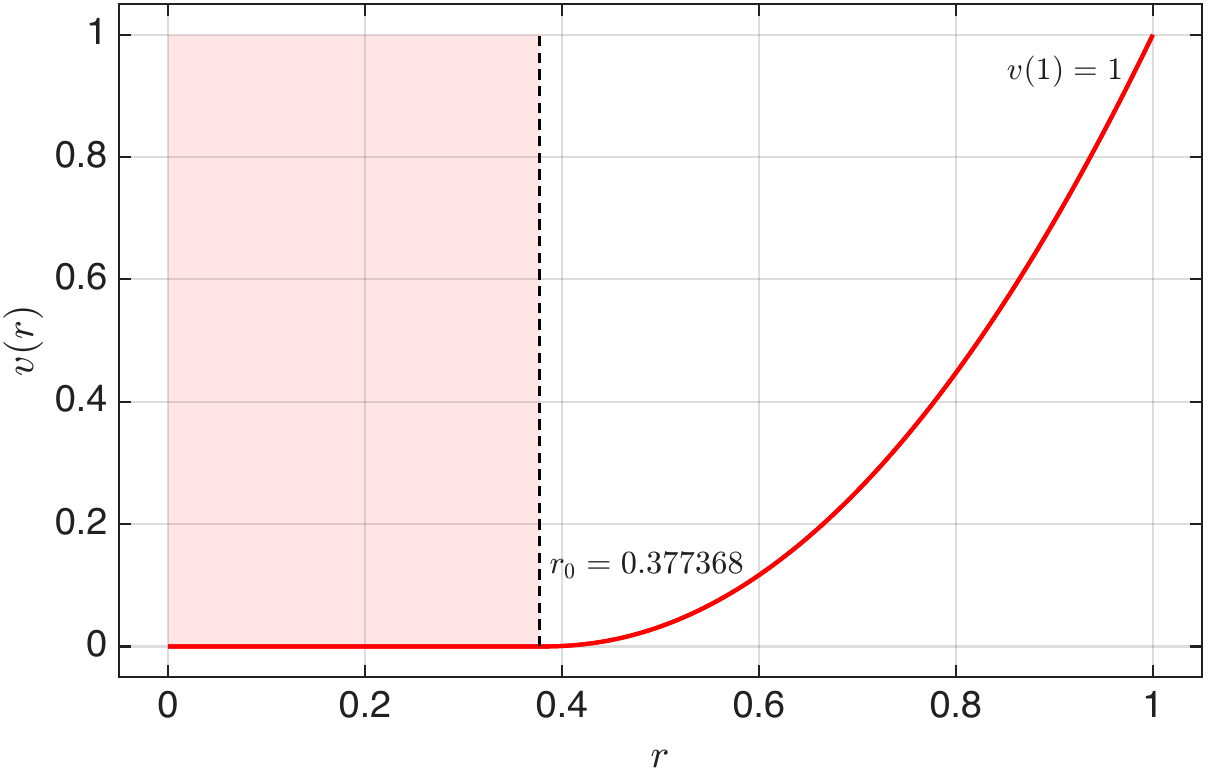}
    \caption{Radial reference solution for the unit-disk benchmark with
    $\varphi=3$ and $p=0.1$. The resulting reference dead-core radius is
    $r_0=0.377368$. The solution is zero for $0\le r\le r_0$ and
    satisfies $v(1)=1$.}
    \label{fig:benchmark_radial_profile}
\end{figure}

\begin{table}[ht]
\centering
\setlength{\tabcolsep}{8pt}
\renewcommand{\arraystretch}{1.15}
\begin{tabular}{c c c c c}
\hline
$p$ & $\varphi$ & $r_0$ & $A_d/A_\Omega$ & $E$ \\
\hline
0.1 & 3 & 0.3773683824 & 0.1424068960 & 14.1953655088 \\
0.2 & 3 & 0.2579839639 & 0.0665557256 & 13.0660678475 \\
0.4 & 7 & 0.5836165830 & 0.3406083159 & 29.0110978949 \\
0.7 & 7 & 0.0589755368 & 0.0034781139 & 24.0236449642 \\
\hline
\end{tabular}
\caption{Reference dead-core radius, dead-core area fraction, and
total energy for different reaction exponents $p$ in the unit-disk
problem. The first row corresponds to the benchmark case
$\varphi=3$, $p=0.1$.}
\label{tab:p_sensitivity}
\end{table}

\subsubsection{Threshold sensitivity of the reference interface}
\label{sec:threshold_sensitivity}

The dead-core boundary is subsequently reconstructed from the numerical
solution by means of a positive detection threshold. To separate the
effect of this post-processing parameter from the finite element
discretization error, we first investigate its influence directly on
the radial reference solution.

Let \(r_\alpha\) denote the radius at which the reference profile
satisfies
\[
v(r_\alpha)=\alpha,
\qquad
\alpha>0.
\]
The dead-core boundary of the reference solution is located at
\(r=r_0\), where
\(
v(r_0)=v'(r_0)=0.
\)
In the positive region, the radial equation is
\[
v''+\frac{1}{r}v'=\varphi^2v^p.
\]
Near the interface, let \(s=r-r_0\) and assume
\[
v(r)\sim C_0s^\gamma,
\qquad s\to0^+.
\]
Since \(r_0>0\), the term \(r^{-1}v'\) is of lower order than
\(v''\). Balancing the leading-order terms gives
\[
C_0\gamma(\gamma-1)s^{\gamma-2}
\sim
\varphi^2 C_0^p s^{\gamma p},
\]
and hence
\[
\gamma-2=\gamma p,
\qquad
\gamma=\frac{2}{1-p}.
\]
Thus,
\[
v(r)\sim C_0(r-r_0)^{2/(1-p)}
\qquad\text{as }r\to r_0^+.
\]
Consequently, imposing \(v(r_\alpha)=\alpha\) gives
\(
r_\alpha-r_0
\sim
C\,\alpha^{(1-p)/2},
\)
and therefore
\begin{equation}
\left|r_\alpha-r_0\right|
=
\mathcal{O}\!\left(
\alpha^{(1-p)/2}
\right).
\label{eq:threshold_error}
\end{equation}

It is convenient to denote the exponent in this scaling by
\(\beta\), so that
\begin{equation}
e_\alpha:=\left|r_\alpha-r_0\right|
\sim C\alpha^\beta,
\qquad
\beta=\frac{1-p}{2}.
\label{eq:threshold_scaling}
\end{equation}
The exponent \(\beta\) in \eqref{eq:threshold_scaling}
characterizes the sensitivity of the reconstructed interface to the
detection threshold, while \(\beta_{\mathrm{fit}}\) is obtained from a
log--log fit of \(e_\alpha\) against \(\alpha\).

The threshold sensitivity is evaluated directly from the
one-dimensional radial reference solutions. Table~\ref{tab:threshold_benchmark}
lists \(r_0\), the interface errors \(e_\alpha\) for three representative
thresholds, and the fitted and theoretical exponents.

\begin{table}[ht]
\centering
\setlength{\tabcolsep}{7pt}
\renewcommand{\arraystretch}{1.15}
\begin{tabular}{c c c c c c c c}
\hline
$p$ & $\varphi$ & $r_0$ &
$e_{10^{-3}}$ & $e_{10^{-5}}$ & $e_{10^{-7}}$ &
$\beta_{\mathrm{fit}}$ & $\beta_{\mathrm{theory}}$ \\
\hline
0.1 & 3 & 0.377368 &
$2.48\times10^{-2}$ & $3.09\times10^{-3}$ & --- &
0.453 & 0.450 \\
0.2 & 3 & 0.257984 &
$4.15\times10^{-2}$ & $6.47\times10^{-3}$ & $1.02\times10^{-3}$ &
0.404 & 0.400 \\
0.4 & 7 & 0.583617 &
$5.05\times10^{-2}$ & $1.26\times10^{-2}$ & $3.16\times10^{-3}$ &
0.301 & 0.300 \\
0.7 & 7 & 0.058976 &
$3.31\times10^{-1}$ & $1.65\times10^{-1}$ & $8.25\times10^{-2}$ &
0.149 & 0.150 \\
\hline
\end{tabular}
\caption{Threshold sensitivity of the dead-core radius evaluated
directly from the one-dimensional radial reference solution for
the unit-disk problem. Here, \(e_\alpha=|r_\alpha-r_0|\),
\(\beta_{\mathrm{fit}}\) is obtained from a log--log fit of
\(e_\alpha\) against \(\alpha\), and
\(\beta_{\mathrm{theory}}=(1-p)/2\). The symbol ``---'' indicates
that the corresponding threshold value was not resolved by the
available numerical sampling of the reference profile.}
\label{tab:threshold_benchmark}
\end{table}

The fitted exponents agree closely with the theoretical values given
by Equation~\eqref{eq:threshold_scaling} for all four parameter sets.
As \(p\) increases, the exponent \(\beta\) decreases, indicating a
slower reduction of the interface error with decreasing threshold.
Thus, for larger values of \(p\), smaller detection thresholds may be
required to achieve a comparable interface accuracy.

For \(p=0.1\) and \(\varphi=3\), the value \(e_{10^{-7}}\) is not
reported because the available numerical sampling of the reference
profile does not resolve the solution sufficiently close to the
dead-core boundary at this threshold. The symbol ``---'' therefore
denotes an unresolved value rather than a zero error.

This one-dimensional study provides a reference for the subsequent
finite element calculations. Since the threshold represents a
post-processing parameter rather than a parameter of the underlying
continuous problem, it is kept fixed across the mesh-refinement levels
in the absence of an independently known interface.

\subsubsection{FEM convergence}

We first assess the convergence of the finite element solution under
successive mesh refinement. Table~\ref{tab:fem-convergence} reports
the energy and the $L^2$ and $H^1$ errors with respect to the radial
reference solution. The $L^2$ error measures the overall magnitude of
the difference between the numerical and reference solution,
whereas the $H^1$-seminorm error measures the difference in their
spatial gradients. Thus, the former primarily assesses the accuracy
of the solution values, while the latter assesses how accurately their
spatial variation is reproduced. In particular,
\[
\|e\|_{L^2}
=
\left(\int_\Omega |u_h-u|^2\,\mathrm{d}x\right)^{1/2},
\qquad
|e|_{H^1}
=
\left(\int_\Omega |\nabla u_h-\nabla u|^2\,\mathrm{d}x\right)^{1/2},
\qquad e=u_h-u.
\]

The error norms are evaluated elementwise using a seven-point Dunavant
quadrature rule, which is exact for polynomial integrands of degree up
to five. This higher-order quadrature provides accurate numerical
integration of the error measures independently of the quadrature used
in the finite element discretization.

\begin{table}
\centering
\setlength{\tabcolsep}{8pt}
\renewcommand{\arraystretch}{1.15}
\begin{tabular}{c c r c c c c}
\hline
level & $h_{\max}$ & $N_T$
& $E_h$ & $|E_h-E_{\mathrm{ref}}|$
& $\|e\|_{L^2}$ & $|e|_{H^1}$ \\
\hline
0 & 0.2000 & 176
  & 14.3153812073 & $1.20016\mathrm{e}{-1}$
  & $4.62231\mathrm{e}{-2}$ & $5.81221\mathrm{e}{-1}$ \\
1 & 0.1000 & 720
  & 14.2226052464 & $2.72397\mathrm{e}{-2}$
  & $6.15071\mathrm{e}{-3}$ & $2.84900\mathrm{e}{-1}$ \\
2 & 0.0500 & 2880
  & 14.2019003926 & $6.53488\mathrm{e}{-3}$
  & $1.54384\mathrm{e}{-3}$ & $1.40515\mathrm{e}{-1}$ \\
3 & 0.0250 & 11584
  & 14.1969881999 & $1.62269\mathrm{e}{-3}$
  & $5.86076\mathrm{e}{-4}$ & $7.00389\mathrm{e}{-2}$ \\
4 & 0.0125 & 46464
  & 14.1957658042 & $4.00295\mathrm{e}{-4}$
  & $1.13631\mathrm{e}{-4}$ & $3.48378\mathrm{e}{-2}$ \\
\hline
Reference & --- & ---
  & 14.1953655088 & --- & --- & --- \\
\hline
\end{tabular}
\caption{FEM convergence for the benchmark problem on the unit disk
with $\varphi=3$ and $p=0.1$. Here, $N_T$ denotes the number of triangular elements,
$E_h$ is the discrete energy, and $e=u_h-u$ denotes the error with
respect to the high-accuracy radial reference solution.}
\label{tab:fem-convergence}
\end{table}

The computed energy approaches the reference value
$E_{\mathrm{ref}}=14.1953655088$ under mesh refinement, with an
observed convergence rate of approximately second order. The
$H^1$-seminorm error converges at approximately first order, as
expected for linear finite elements. The $L^2$ error decreases under
mesh refinement, with observed rates approaching second order on the
finer levels.

\subsubsection{Dead-core radius and interface reconstruction}

We now examine the reconstruction of the dead-core interface from the
finite element solution. The interface is identified as the level set
\(
u=\alpha,
\)
where $\alpha>0$ is used only in the post-processing step. To assess
the sensitivity of the reconstruction, the same three absolute
thresholds,
\[
\alpha\in\{10^{-3},10^{-4},10^{-5}\},
\]
are applied at every mesh level. No threshold is selected by comparison
with the reference interface.

\begin{table}
\centering
\setlength{\tabcolsep}{8pt}
\renewcommand{\arraystretch}{1.15}
\begin{tabular}{c c c c c}
\hline
level & $h_{\max}$ &
$r_h(\alpha=10^{-3})$ &
$r_h(\alpha=10^{-4})$ &
$r_h(\alpha=10^{-5})$ \\
\hline
0 & 0.2000 & 0.396500 & 0.394925 & 0.394585 \\
1 & 0.1000 & 0.402489 & 0.399060 & 0.398478 \\
2 & 0.0500 & 0.404676 & 0.399876 & 0.359734 \\
3 & 0.0250 & 0.403482 & 0.379194 & 0.370564 \\
4 & 0.0125 & 0.401987 & 0.386128 & 0.376096 \\
\hline
Reference & --- &
\multicolumn{3}{c}{$0.377368$} \\
\hline
\end{tabular}
\caption{Reconstructed dead-core radius for the unit-disk benchmark
using fixed absolute threshold values
$\alpha=10^{-3}$, $10^{-4}$, and $10^{-5}$.}
\label{tab:dead-core-alpha}
\end{table}

Table~\ref{tab:dead-core-alpha} demonstrates the dependence of the
reconstructed interface on the threshold parameter. The largest
threshold produces a persistent bias in the reconstructed interface,
whereas smaller thresholds give closer agreement with the reference
solution on the finer meshes.

The results for the smallest threshold also reveal an interaction
between threshold selection and spatial resolution. In particular, the
reconstructed radius is not strictly monotone with mesh refinement.
This behavior indicates that very small threshold values are more
sensitive to the numerical resolution of the computed solution in the
neighborhood of the dead-core boundary.

Table~\ref{tab:dead-core-accuracy} shows the corresponding accuracy of
the interface and dead-core area reconstruction. For the unit disk, the reconstructed dead-core area fraction is
obtained directly from the reconstructed radius as
\(
\frac{A_{d,h}}{A_\Omega}=r_h^2.
\)
The radius and area
fraction both approach the reference values under refinement, with
the finest mesh providing a close approximation of the dead-core
geometry.

\begin{table}[H]
\centering
\setlength{\tabcolsep}{9pt}
\renewcommand{\arraystretch}{1.15}
\begin{tabular}{c c c c c c}
\hline
level & $h_{\max}$ & $r_h$ &
$|r_h-r_0|$ &
$A_{d,h}/A_\Omega$ &
$\left|A_{d,h}/A_\Omega-A_d/A_\Omega\right|$ \\
\hline
0 & 0.2000 & 0.394585 & 1.722e-2 & 0.155697 & 1.329e-2 \\
1 & 0.1000 & 0.398478 & 2.111e-2 & 0.158784 & 1.638e-2 \\
2 & 0.0500 & 0.359734 & 1.763e-2 & 0.129408 & 1.300e-2 \\
3 & 0.0250 & 0.370564 & 6.804e-3 & 0.137318 & 5.089e-3 \\
4 & 0.0125 & 0.376096 & 1.272e-3 & 0.141448 & 9.588e-4 \\
\hline
Reference & --- & 0.377368 & --- & 0.142407 & --- \\
\hline
\end{tabular}
\caption{Accuracy of the reconstructed dead-core radius and area
fraction for the unit-disk benchmark using the fixed threshold
$\alpha=10^{-5}$. For the unit disk,
$A_{d,h}/A_\Omega=r_h^2$.}
\label{tab:dead-core-accuracy}
\end{table}

Overall, the results confirm that the threshold parameter has a
significant influence on the reconstructed free boundary. It should
therefore be regarded as a post-processing parameter whose choice
needs to be considered together with the spatial resolution. The
results also indicate that sufficiently fine spatial resolution is
important when small threshold values are used for interface
reconstruction.

\subsection{Unit disk with a concentric hole}

As a modification of the benchmark problem, we consider a unit disk
containing a concentric circular hole of radius \(1/4\). 

\begin{figure}[H]
\centering
\begin{minipage}[t]{0.47\linewidth}
    \centering
    \includegraphics[width=\linewidth]{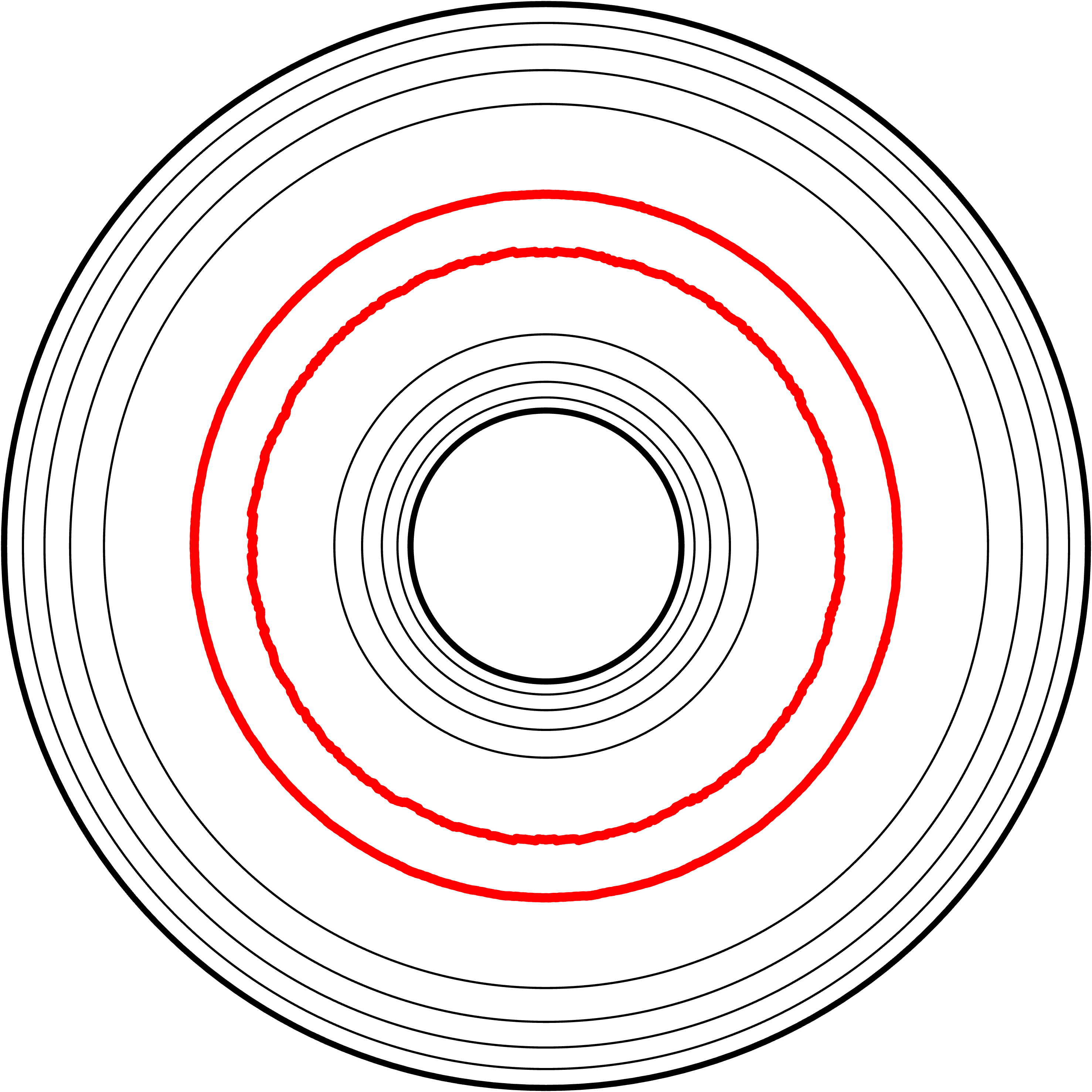}
\end{minipage}
\hfill
\begin{minipage}[t]{0.47\linewidth}
    \centering
    \includegraphics[width=\linewidth]{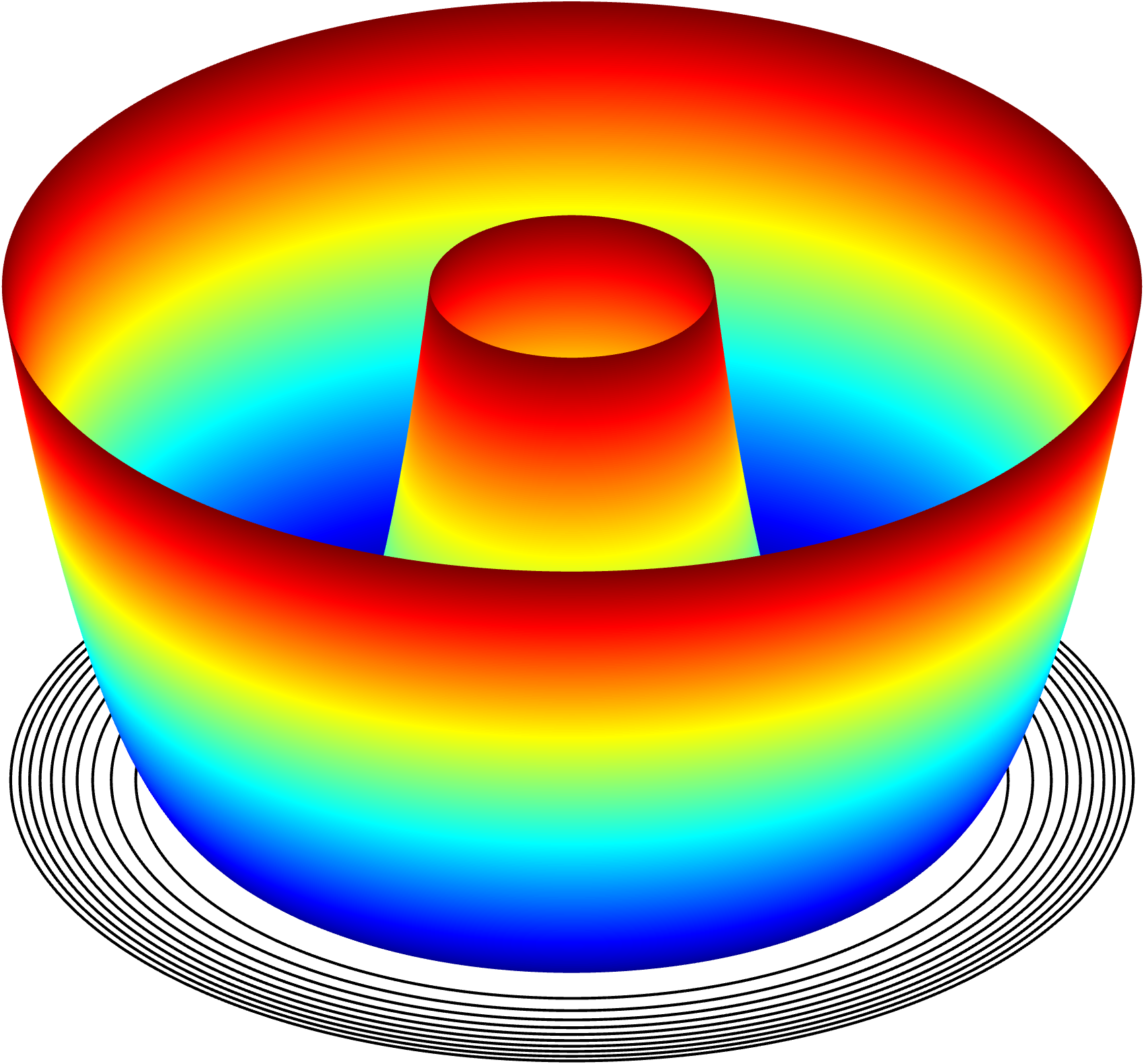}
\end{minipage}
\caption{Finite element approximation on the disk with a concentric hole for
$\varphi=5$ and $p=0.1$. Left: contour plot with the reconstructed
dead-core interface using $\alpha=10^{-5}$, shown in red. Right:
three-dimensional visualization of the numerical solution.}
\label{fig:diskhole_deadcore}
\end{figure}

The dead zone consists of two concentric circular interfaces, as shown
in Figure~\ref{fig:diskhole_deadcore}, and no isolated dead-zone
islands occur. This follows from radial symmetry of the solution.

\begin{lemma}
Let
\[
\Omega=\left\{(x,y)\in\mathbb R^2:
r_1<\sqrt{x^2+y^2}<1\right\},
\qquad 0<r_1<1.
\]
Then the dead-zone is necessarily a concentric annulus.
\end{lemma}

\begin{proof}
By rotational symmetry and the imposed boundary conditions, the solution
to Eq.~\eqref{reactionvolume} is radially symmetric,
\(
u(x,y)=v(r),
r=\sqrt{x^2+y^2},
\)
where \(v\) satisfies
\[
-\frac1r(r v')'+\varphi^2[v]_+^p=0
\quad\text{in }(r_1,1),
\qquad
v(r_1)=v(1)=1.
\]

Since \(v(r_1)=v(1)=1\), Rolle’s theorem implies the existence of
\(r_m\in(r_1,1)\) such that
\[
v'(r_m)=0.
\]

Furthermore,
\[
(rv')'=\varphi^2 r[v]_+^p\ge0,
\]
and therefore \(rv'\) is nondecreasing on \((r_1,1)\). Since \(r>0\),
it follows that
\[
v'(r)\le0 \quad \text{for } r<r_m,
\qquad
v'(r)\ge0 \quad \text{for } r>r_m.
\]

Hence \(v\) is nonincreasing on \((r_1,r_m)\) and nondecreasing on
\((r_m,1)\). Assume now that there exist \(r_a<r_b\) such that
\[
v(r_a)=v(r_b)=0.
\]
Since \(v\ge0\), the monotonicity properties imply
\[
v(r)=0
\qquad
\forall r\in[r_a,r_b].
\]
Thus the zero set of \(v\) is an interval, and the dead-zone therefore
forms a concentric annulus.
\end{proof}

No analytical reference solution is available for the disk-with-hole
geometry. Therefore, we first investigate the convergence of the
computed discrete energies. 
Table~\ref{tab:diskhole-energy} shows a
monotone decrease of the discrete energy under mesh refinement,
indicating convergence of the finite element approximation. 

\begin{table}[H]
\centering
\setlength{\tabcolsep}{8pt}
\renewcommand{\arraystretch}{1.15}
\begin{tabular}{c c r c}
\hline
level & $h_{\max}$ & $N_T$ & $E_h$ \\
\hline
0 & 0.2000 & 236   & 35.9356509677 \\
1 & 0.1000 & 680   & 34.6200750680 \\
2 & 0.0500 & 2724  & 34.0365841901 \\
3 & 0.0250 & 11212 & 33.8992278692 \\
4 & 0.0125 & 42900 & 33.8650711409 \\
\hline
\end{tabular}
\caption{Discrete energy under mesh refinement for the
disk-with-hole geometry.}
\label{tab:diskhole-energy}
\end{table}

Assuming asymptotic second-order convergence of the discrete energy,
a Richardson extrapolation based on the two finest meshes gives
\[
E_{\mathrm{ref}}
\approx
E_{h_4}
+
\frac{E_{h_4}-E_{h_3}}{2^2-1},
\]
where $E_{h_3}$ and $E_{h_4}$ denote the discrete energies
corresponding to levels $3$ and $4$, respectively. This gives
\[
E_{\mathrm{ref}} \approx 33.8536855648.
\]

We next investigate the reconstruction of the dead-zone interface
under mesh refinement. The interface is identified as the level set
\(
u=\alpha,
\)
where $\alpha>0$ is a post-processing detection threshold. Several
fixed values of $\alpha$ are considered to assess the sensitivity of
the reconstructed interface to this parameter.

Table~\ref{tab:diskhole-alpha} reports the reconstructed inner and
outer interface radii. The results show that the detected interface
depends on both the mesh resolution and the selected threshold. The
larger threshold gives a relatively stable reconstruction across the
mesh levels, whereas smaller thresholds exhibit a stronger dependence
on the spatial resolution. On the coarsest mesh, the smallest
thresholds are not resolved and therefore do not produce a
reconstructed interface.

\begin{table}[H]
\centering
\setlength{\tabcolsep}{6pt}
\renewcommand{\arraystretch}{1.15}
\begin{tabular}{c c c c c c c c}
\hline
level & $h_{\max}$ &
\multicolumn{2}{c}{$\alpha=10^{-3}$} &
\multicolumn{2}{c}{$\alpha=10^{-4}$} &
\multicolumn{2}{c}{$\alpha=10^{-5}$} \\
& &
$r_{\mathrm{inner}}$ & $r_{\mathrm{outer}}$ &
$r_{\mathrm{inner}}$ & $r_{\mathrm{outer}}$ &
$r_{\mathrm{inner}}$ & $r_{\mathrm{outer}}$ \\
\hline
0 & 0.2000 & 0.517366 & 0.667643 & --- & --- & --- & --- \\
1 & 0.1000 & 0.530553 & 0.666389 & 0.540378 & 0.658178 & 0.542061 & 0.656595 \\
2 & 0.0500 & 0.522214 & 0.669022 & 0.544544 & 0.642869 & 0.564772 & 0.618734 \\
3 & 0.0250 & 0.526294 & 0.668145 & 0.535538 & 0.658764 & 0.538537 & 0.654080 \\
4 & 0.0125 & 0.527844 & 0.665888 & 0.538133 & 0.654615 & 0.542874 & 0.649128 \\
\hline
\end{tabular}
\caption{Reconstructed inner and outer dead-zone interface radii for
selected detection thresholds $\alpha$ under mesh refinement in the
disk-with-hole geometry. A dash indicates that the corresponding
level set was not resolved.}
\label{tab:diskhole-alpha}
\end{table}

The corresponding dead-zone area fractions are reported in
Table~\ref{tab:diskhole-fraction}. Their variation with mesh refinement is
consistent with the behavior observed for the individual interface
radii. In particular, decreasing the detection threshold does not
necessarily lead to a monotone improvement under mesh refinement.
This indicates that the threshold should be chosen in relation to the
spatial resolution, since very small values probe the numerical
solution increasingly close to the dead-zone boundary.

\begin{table}[H]
\centering
\setlength{\tabcolsep}{7pt}
\renewcommand{\arraystretch}{1.15}
\begin{tabular}{c c c c c}
\hline
level & $h_{\max}$ &
$\alpha=10^{-3}$ &
$\alpha=10^{-4}$ &
$\alpha=10^{-5}$ \\
\hline
0 & 0.2000 & 0.189951 & ---      & ---      \\
1 & 0.1000 & 0.173427 & 0.150602 & 0.146439 \\
2 & 0.0500 & 0.186542 & 0.124536 & 0.068122 \\
3 & 0.0250 & 0.180728 & 0.156981 & 0.146985 \\
4 & 0.0125 & 0.175773 & 0.148196 & 0.135098 \\
\hline
\end{tabular}
\caption{Reconstructed dead-zone area fraction
$A_{d,h}/A_\Omega$ for selected detection thresholds $\alpha$
under mesh refinement in the disk-with-hole geometry.}
\label{tab:diskhole-fraction}
\end{table}

%====================================================================

\subsection{Hexagonal geometries}

We next consider reactor geometries based on a regular outer hexagon
with either a hexagonal or a circular hole in the center. In both
cases, homogeneous Dirichlet boundary conditions $u=1$ are imposed on
all boundary components.

The formation and subsequent merging of dead zones depend on the
Thiele modulus, the reaction exponent, and the geometry of the central
hole. As the Thiele modulus increases, a dead zone may first appear as
one or several disconnected regions. With further increase of the
Thiele modulus, these regions may expand and merge, resulting in a
connected dead zone. The corresponding transition values depend on
both the reaction exponent and the geometry of the central hole.

Figure~\ref{fig:hexagonal_geometries} illustrates this behavior for
$p=0.1$ and both types of central hole. For $\varphi=4.8$,
disconnected dead-zone islands are observed in both geometries. At
$\varphi=5$, the dead zone has already become connected for the
hexagonal hole, whereas disconnected islands persist for the circular
hole. For $\varphi=6$, the dead zone is connected in both geometries.
These results demonstrate that, for this reaction exponent, the
transition between disconnected and connected dead zones depends not
only on the Thiele modulus but also on the shape of the central hole.
In all cases, the reconstructed dead-zone interfaces are obtained using
the detection threshold $\alpha=10^{-4}$.

\begin{figure}
\centering

\begin{minipage}[t]{0.48\linewidth}
\centering
\small

\includegraphics[width=0.9\linewidth]{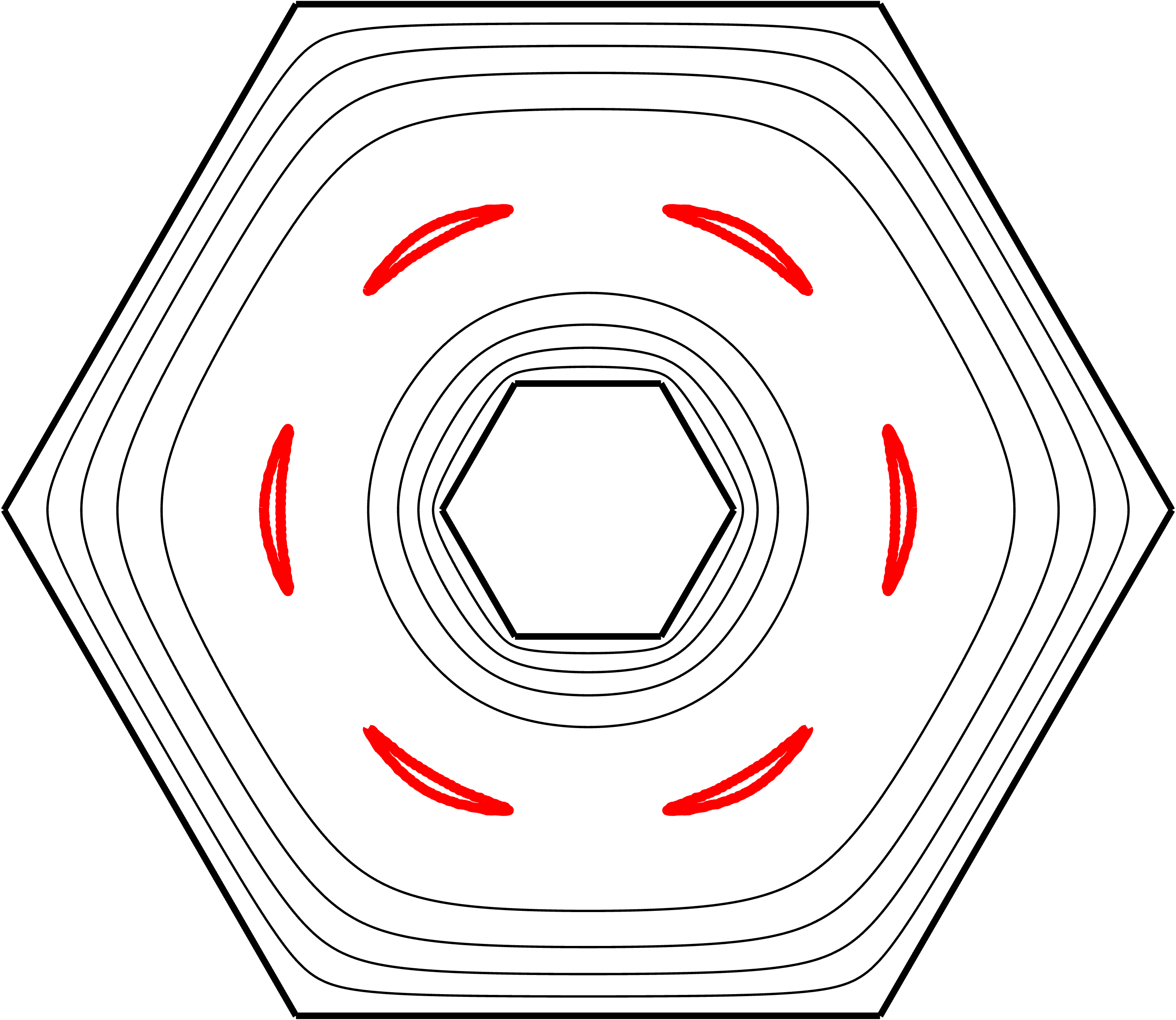}
\vspace{0.2cm}
Disconnected dead-zone islands, $\varphi=4.8$

\vspace{0.4cm}
\includegraphics[width=0.9\linewidth]{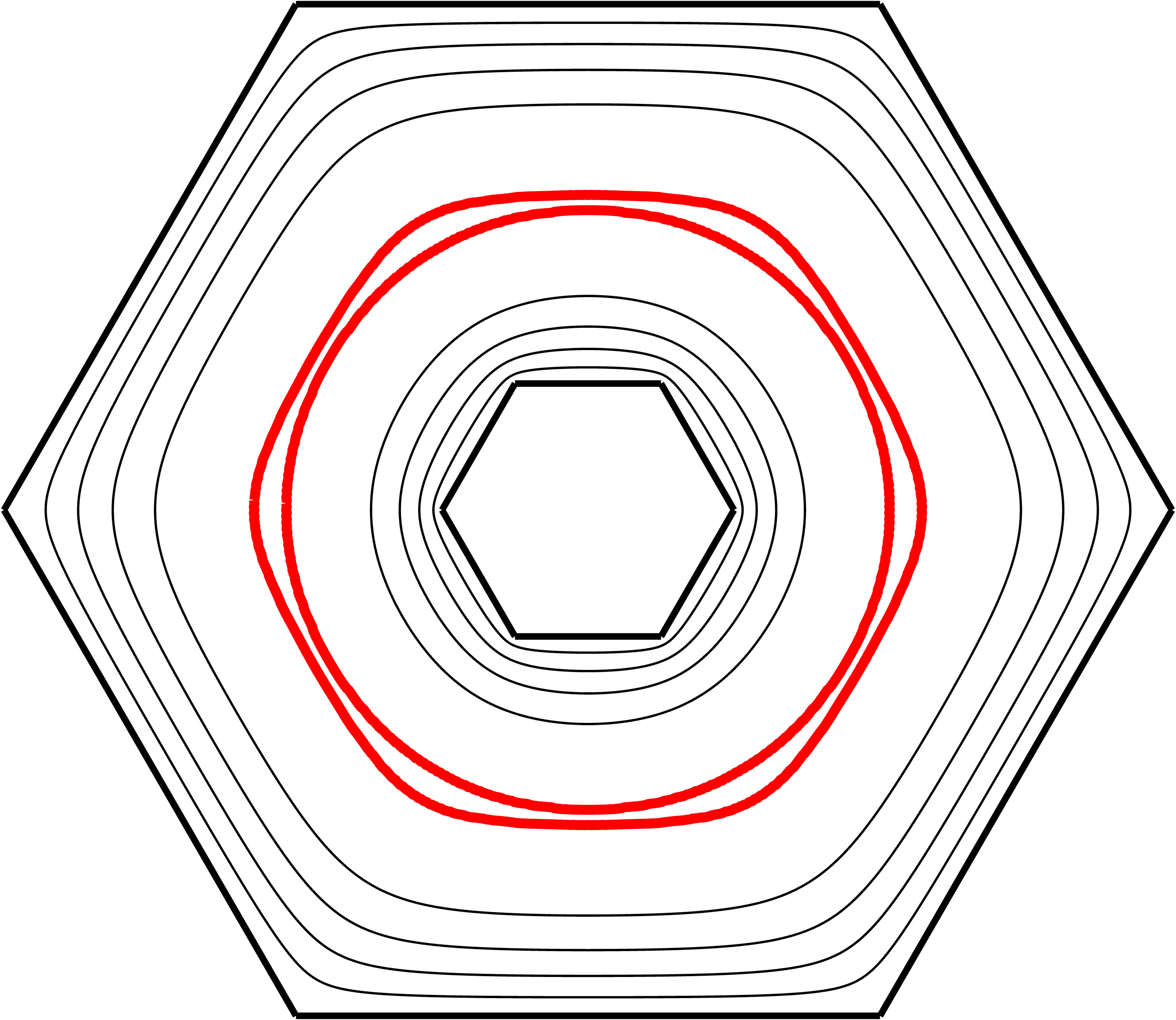}
\vspace{0.2cm}
Connected dead zone, $\varphi=5$

\vspace{0.4cm}
\includegraphics[width=0.9\linewidth]{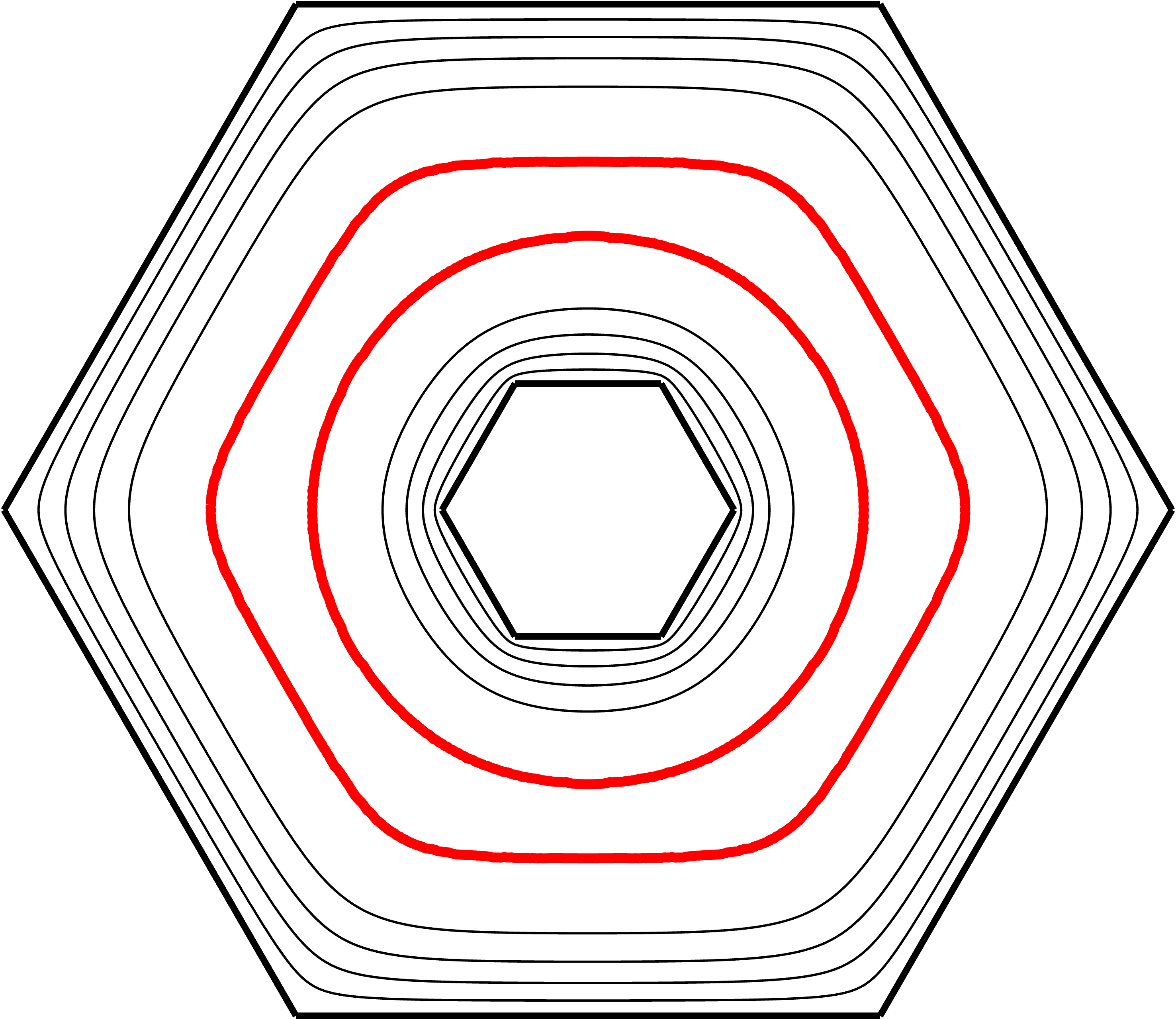}
\vspace{0.2cm}
Connected dead zone, $\varphi=6$

\vspace{0.2cm}
(a) Hexagonal hole
\end{minipage}
\hfill
\begin{minipage}[t]{0.48\linewidth}
\centering
\small

\includegraphics[width=0.9\linewidth]{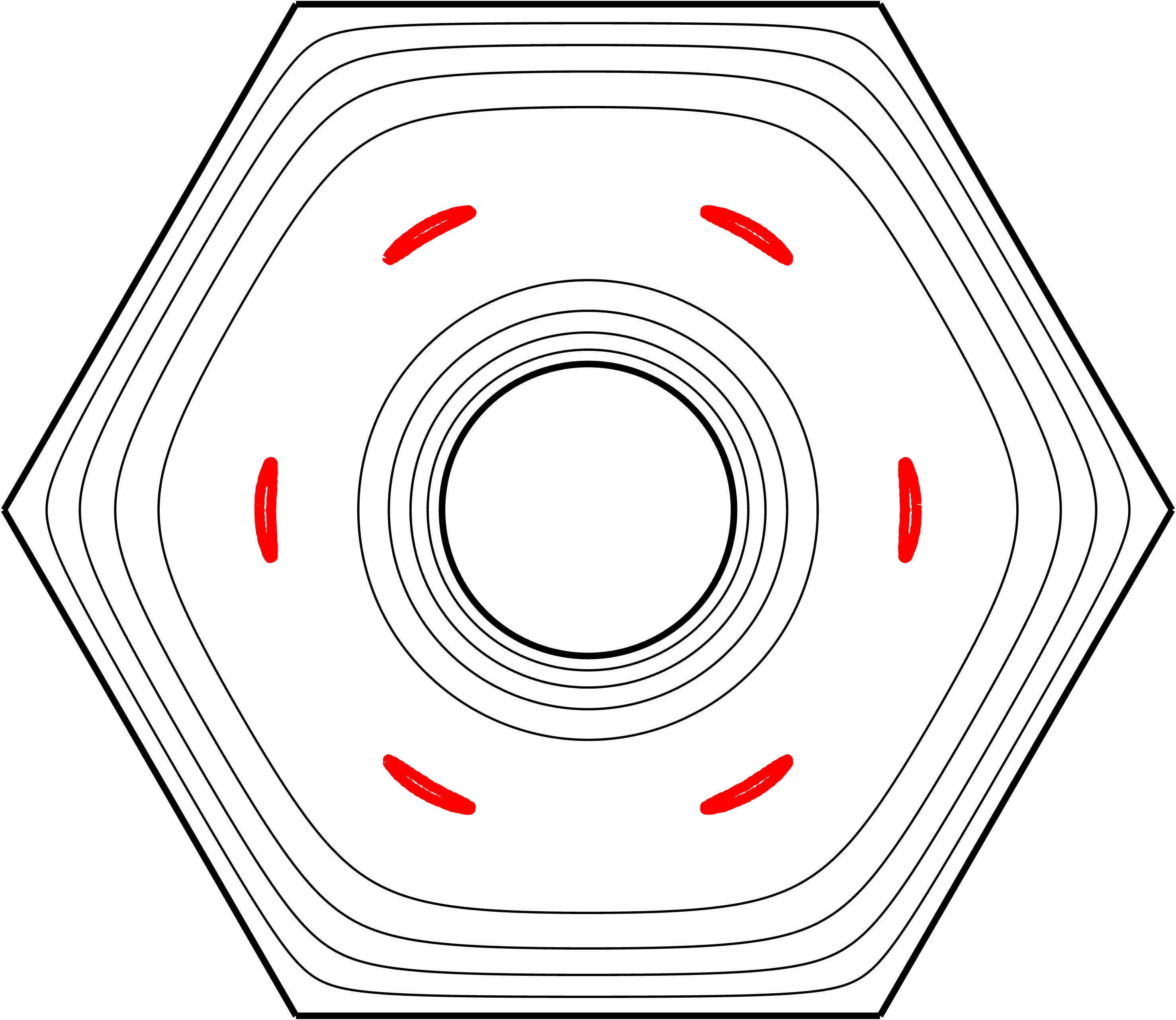}
\vspace{0.2cm}
Disconnected dead-zone islands, $\varphi=4.8$

\vspace{0.4cm}
\includegraphics[width=0.9\linewidth]{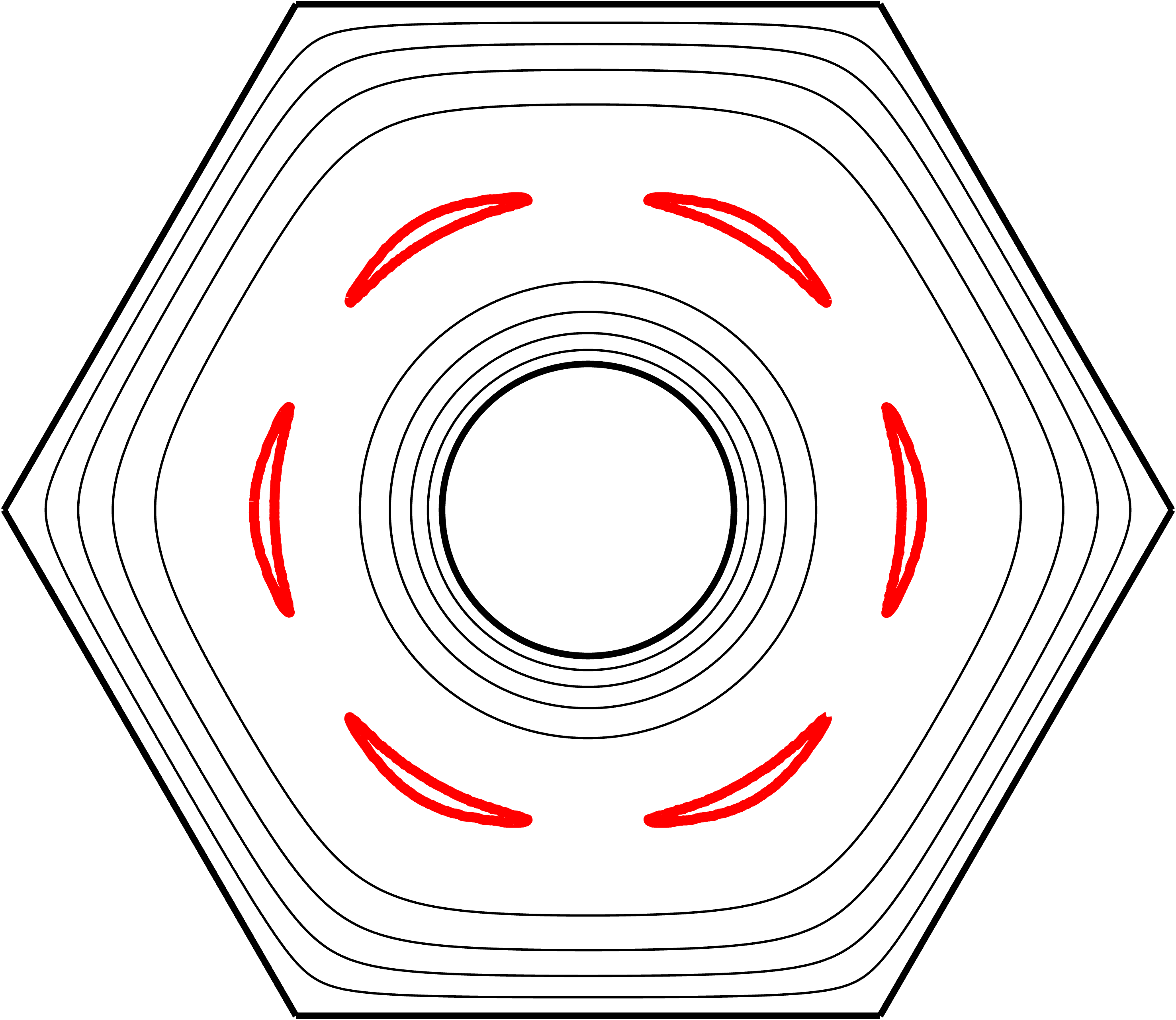}
\vspace{0.2cm}
Disconnected dead zone, $\varphi=5$

\vspace{0.4cm}
\includegraphics[width=0.9\linewidth]{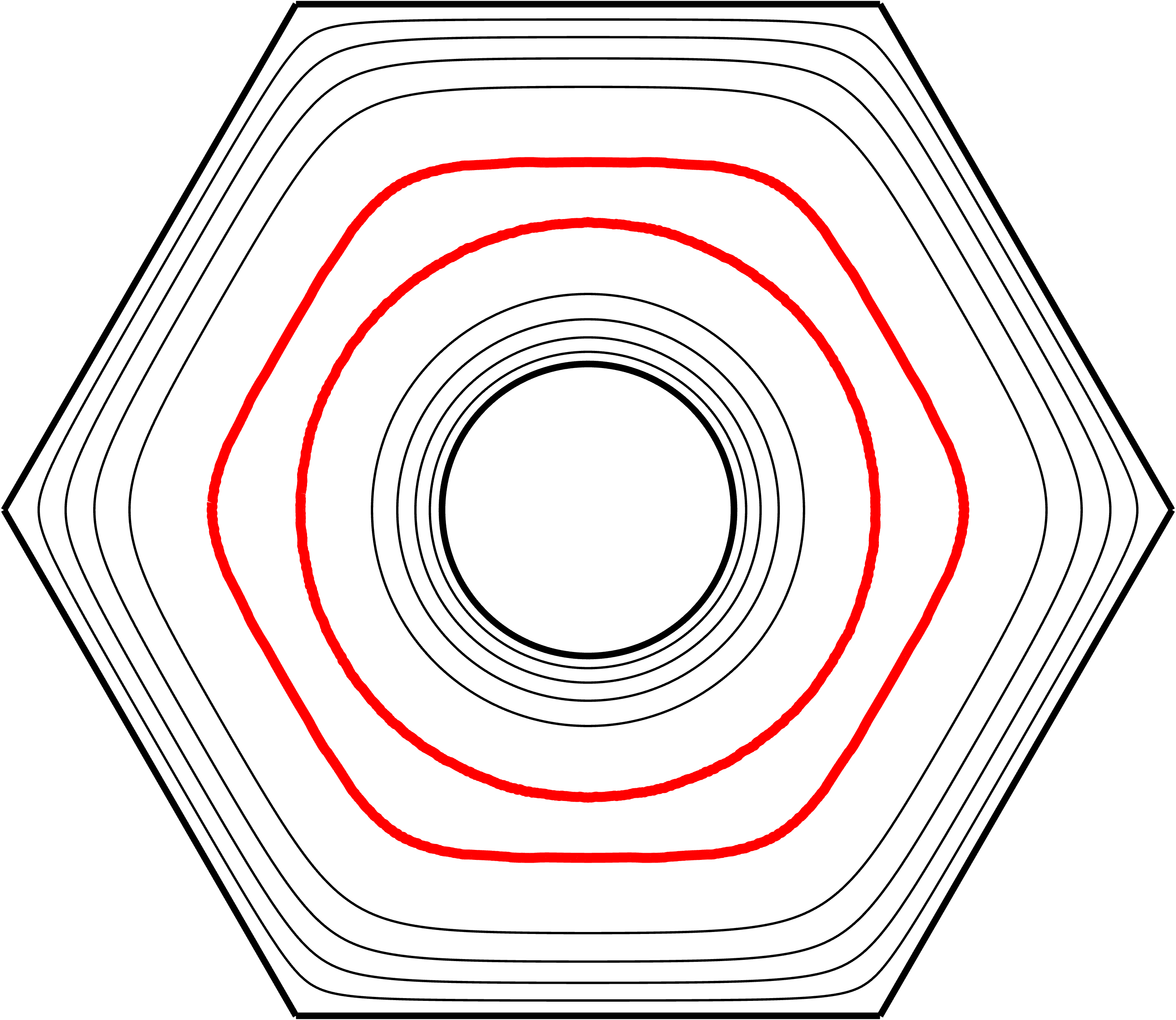}
\vspace{0.2cm}
Connected dead zone, $\varphi=6$

\vspace{0.2cm}
(b) Circular hole
\end{minipage}

\caption{Dead-zone formation in hexagonal geometries for $p=0.1$.
%Results are shown for a hexagonal hole (a) and a circular hole (b) at $\varphi=4.8$, $5$, and $6$. 
Reconstructed dead-zone interfaces are
shown in red using $\alpha=10^{-4}$.}
\label{fig:hexagonal_geometries}
\end{figure}

Figure~\ref{fig:hexagonal_geometries_p} illustrates the effect of the
reaction exponent $p$ on the dead-zone topology at a fixed Thiele
modulus, $\varphi=9.5$. For $p=0.2$, the dead zone forms a connected
region in both the hexagonal-hole and circular-hole geometries. In
contrast, for $p=0.5$, six disconnected dead-zone islands are observed
in both geometries. Thus, increasing the reaction exponent from
$p=0.2$ to $p=0.5$ changes the dead-zone topology at the same value of
the Thiele modulus. The results also indicate that the transition
between connected and disconnected configurations depends on the
reaction exponent and, potentially, on the geometry of the central
hole.

\begin{figure}
\centering

\begin{minipage}[t]{0.48\linewidth}
\centering
\small

\includegraphics[width=0.9\linewidth]{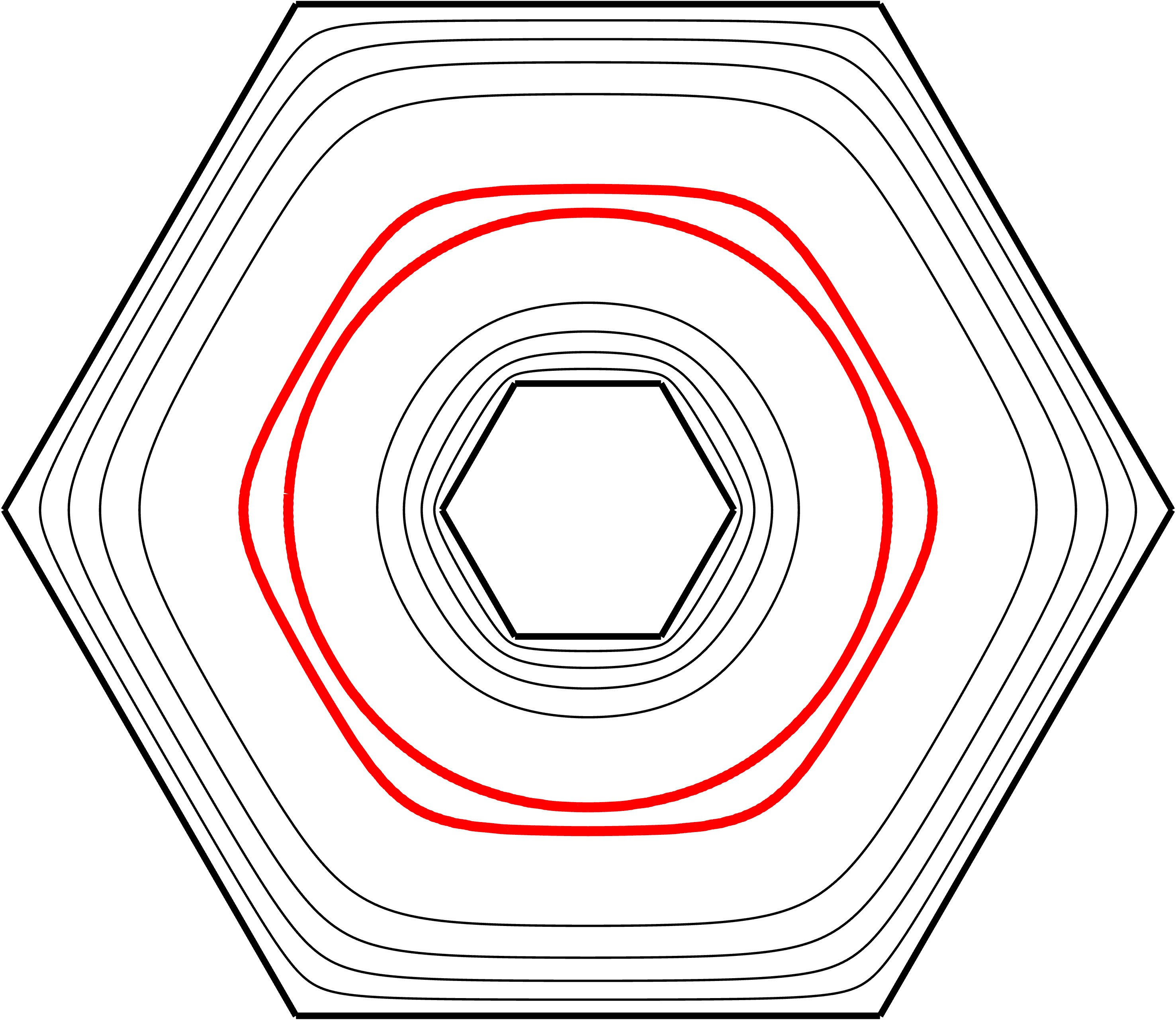}
\vspace{0.2cm}
Connected dead zone, $p=0.2$

\vspace{0.4cm}
\includegraphics[width=0.9\linewidth]{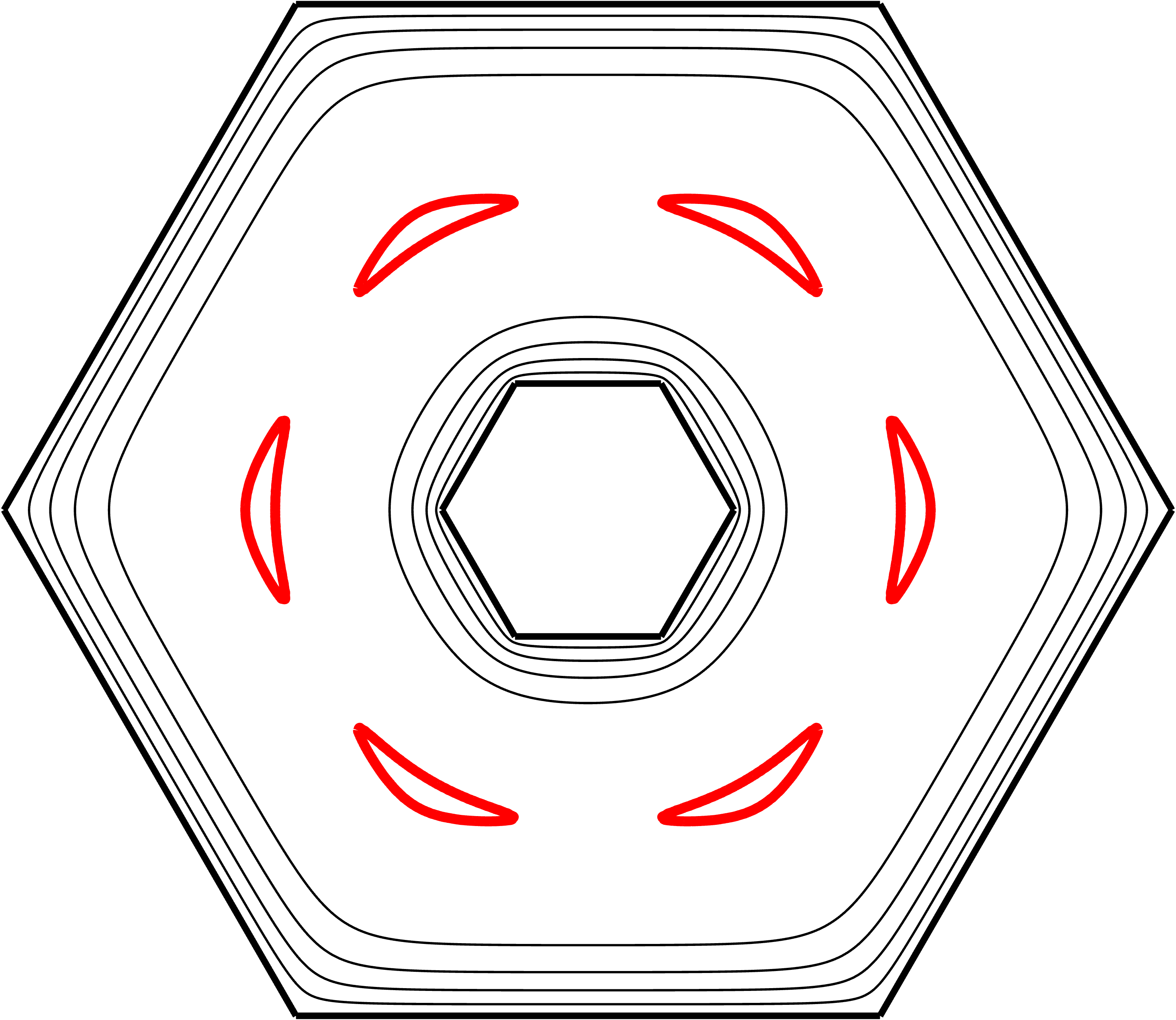}
\vspace{0.2cm}
Disconnected dead-zone islands, $p=0.5$

\vspace{0.2cm}
(a) Hexagonal hole
\end{minipage}
\hfill
\begin{minipage}[t]{0.48\linewidth}
\centering
\small

\includegraphics[width=0.9\linewidth]{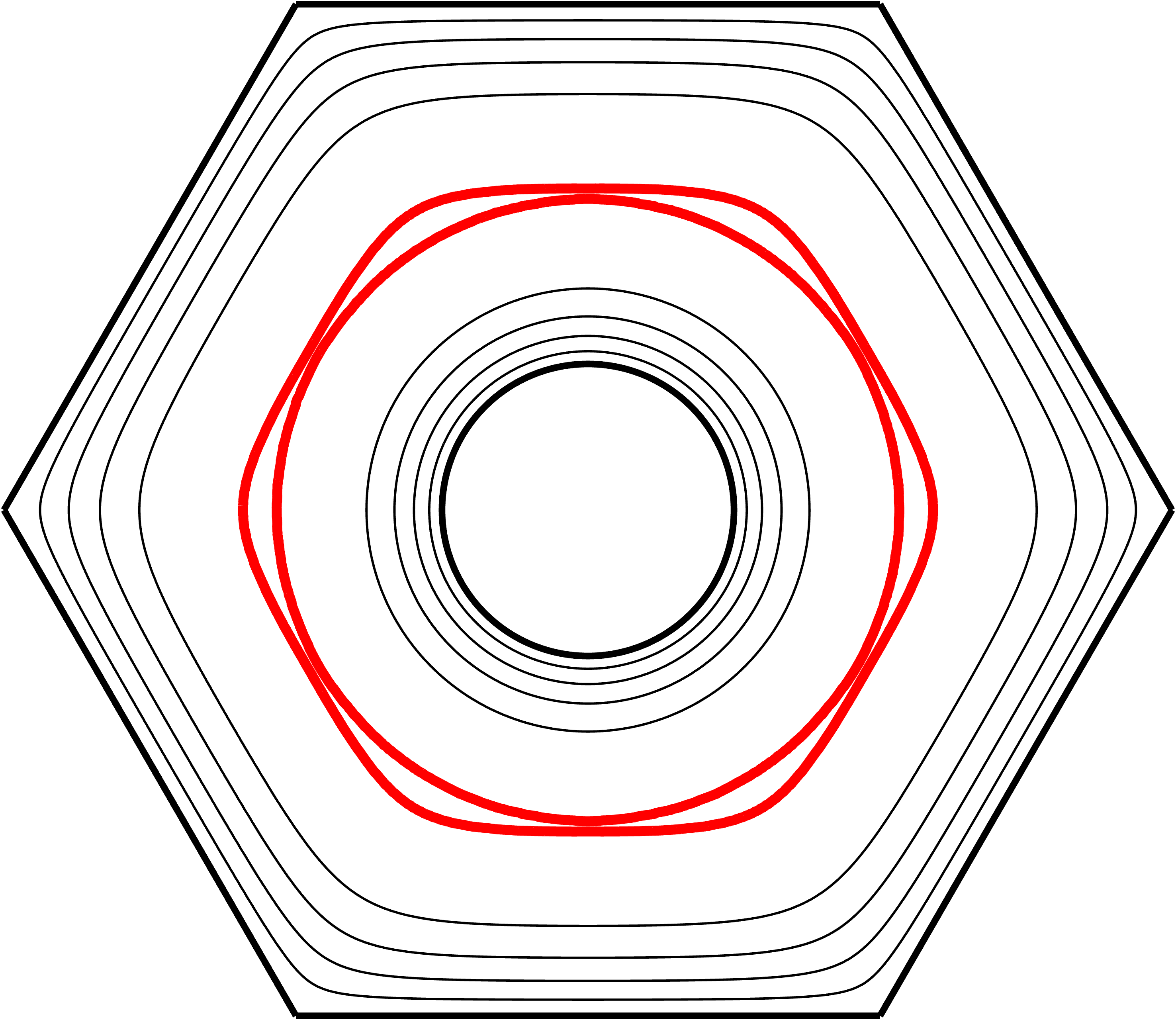}
\vspace{0.2cm}
Connected dead zone, $p=0.2$

\vspace{0.4cm}
\includegraphics[width=0.9\linewidth]{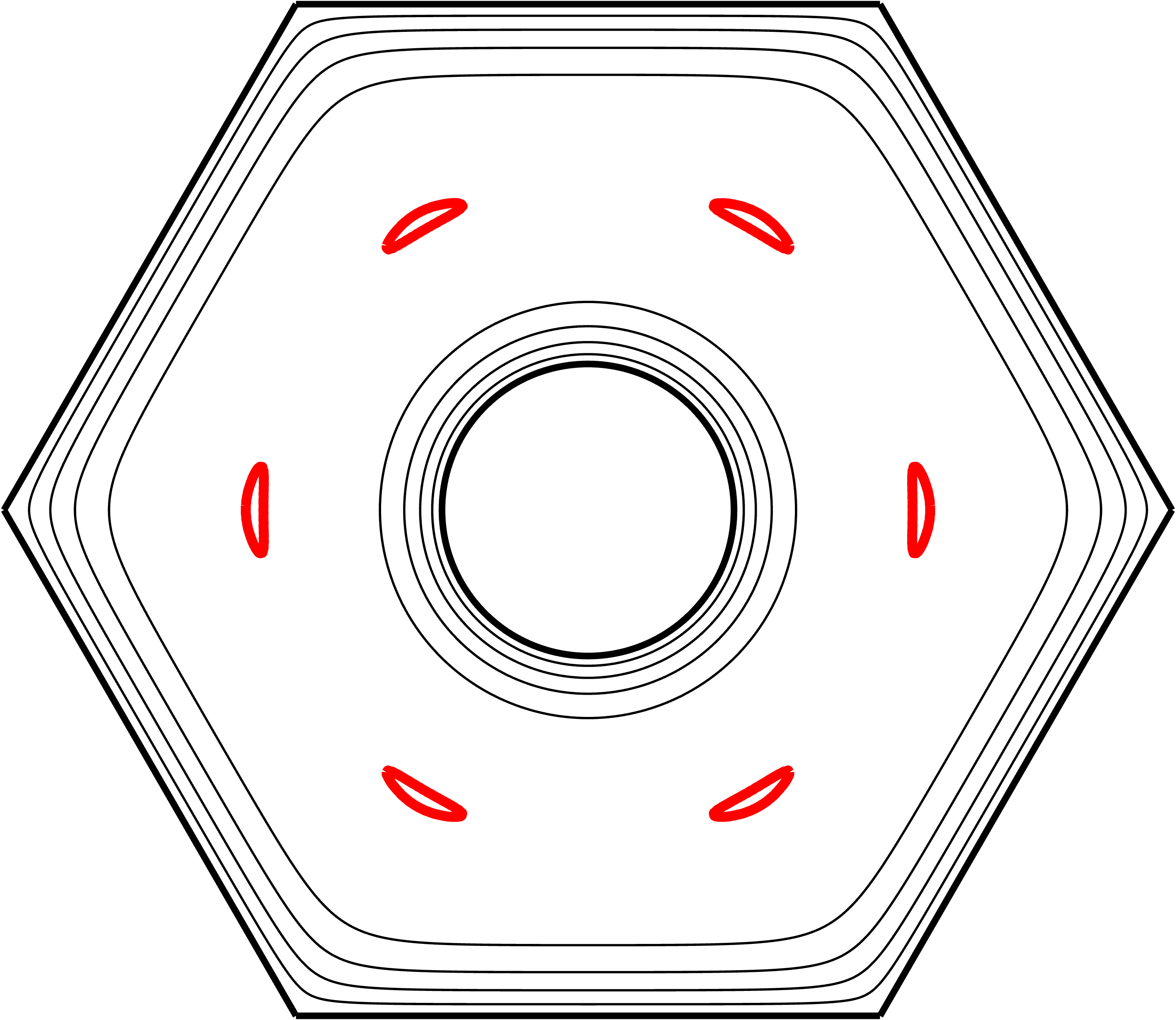}
\vspace{0.2cm}
Disconnected dead-zone islands, $p=0.5$

\vspace{0.2cm}
(b) Circular hole
\end{minipage}

\caption{Dead-zone formation in hexagonal geometries for
$\varphi=9.5$. Reconstructed dead-zone interfaces are shown in red
using $\alpha=10^{-4}$.}
\label{fig:hexagonal_geometries_p}
\end{figure}

%====================================================================

\subsection{Circular geometry with four holes}
Finally, we consider a circular reactor containing four symmetrically placed circular holes. As in the previous examples, the topology of the dead zone is strongly influenced by the Thiele modulus. 

For $\varphi=7$, the dead zone consists of five disconnected components, as shown in Figure~\ref{fig:fourholes_deadcore}(a). The corresponding three-dimensional solution profile is presented in Figure~\ref{fig:fourholes_deadcore}(b). As the Thiele modulus increases to $\varphi=10$, the dead-zone components expand and merge into a connected cross-shaped region; see Figure~\ref{fig:fourholes_deadcore}(c,d). A further increase to $\varphi=13$ leads to substantial growth of the dead zone, which occupies most of the reactor interior and leaves only narrow active regions near the outer boundary and the boundaries of the holes; see Figure~\ref{fig:fourholes_deadcore}(e,f). The three-dimensional visualizations clearly demonstrate the progressive localization of the solution as the Thiele modulus increases. The obtained numerical results are consistent with the asymptotic analysis of Friedman and Phillips \cite{friedman84}, which predicts that the distance between the dead-core interface and the domain boundary scales as $\frac{\gamma}{\varphi}+\bigO\left(\frac{1}{\varphi^2}\right)$ where $\gamma$ is a geometry-dependent positive constant.

The simulation results show  that dead-core formation depends strongly on reactor geometry, the Thiele modulus, and the reaction exponent. In domains containing holes, dead zones tend to nucleate in regions that are furthest from the external boundary and therefore experience the strongest diffusion limitation. Geometric symmetry is reflected in the symmetry of the resulting dead-core islands. As the Thiele modulus increases, individual dead zones expand and eventually merge, producing large connected inactive regions. This behavior resembles geometric percolation phenomena and suggests that reactor topology may strongly affect catalyst utilization.
\begin{figure}
\centering

\begin{minipage}[t]{0.495\linewidth}
\centering
\includegraphics[width=0.8\linewidth]{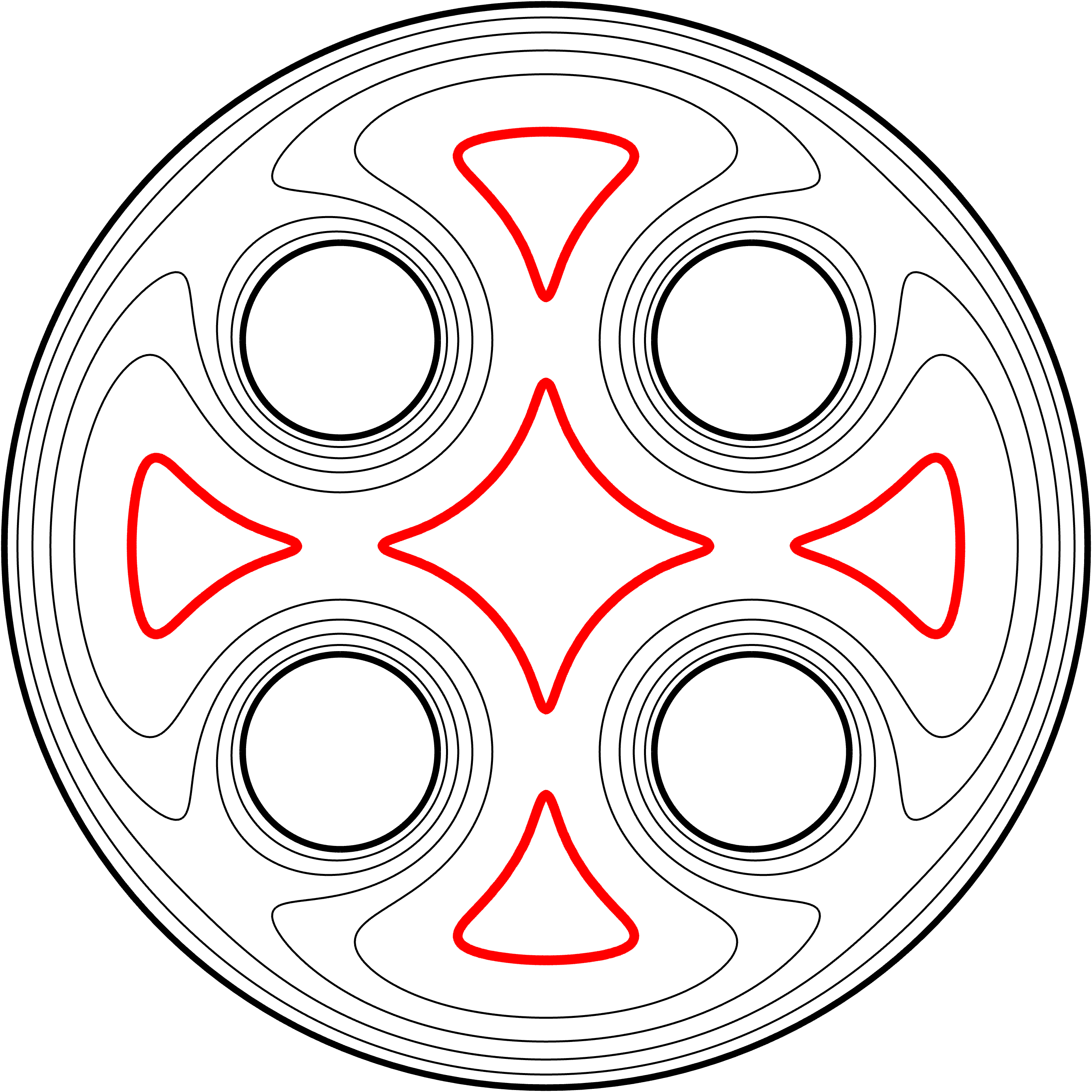}
\vspace{0.02cm}

\small
(a) Disconnected dead-zone components, $\varphi=7$
\end{minipage}
\hfill
\begin{minipage}[t]{0.495\linewidth}
\centering
\includegraphics[width=0.85\linewidth]{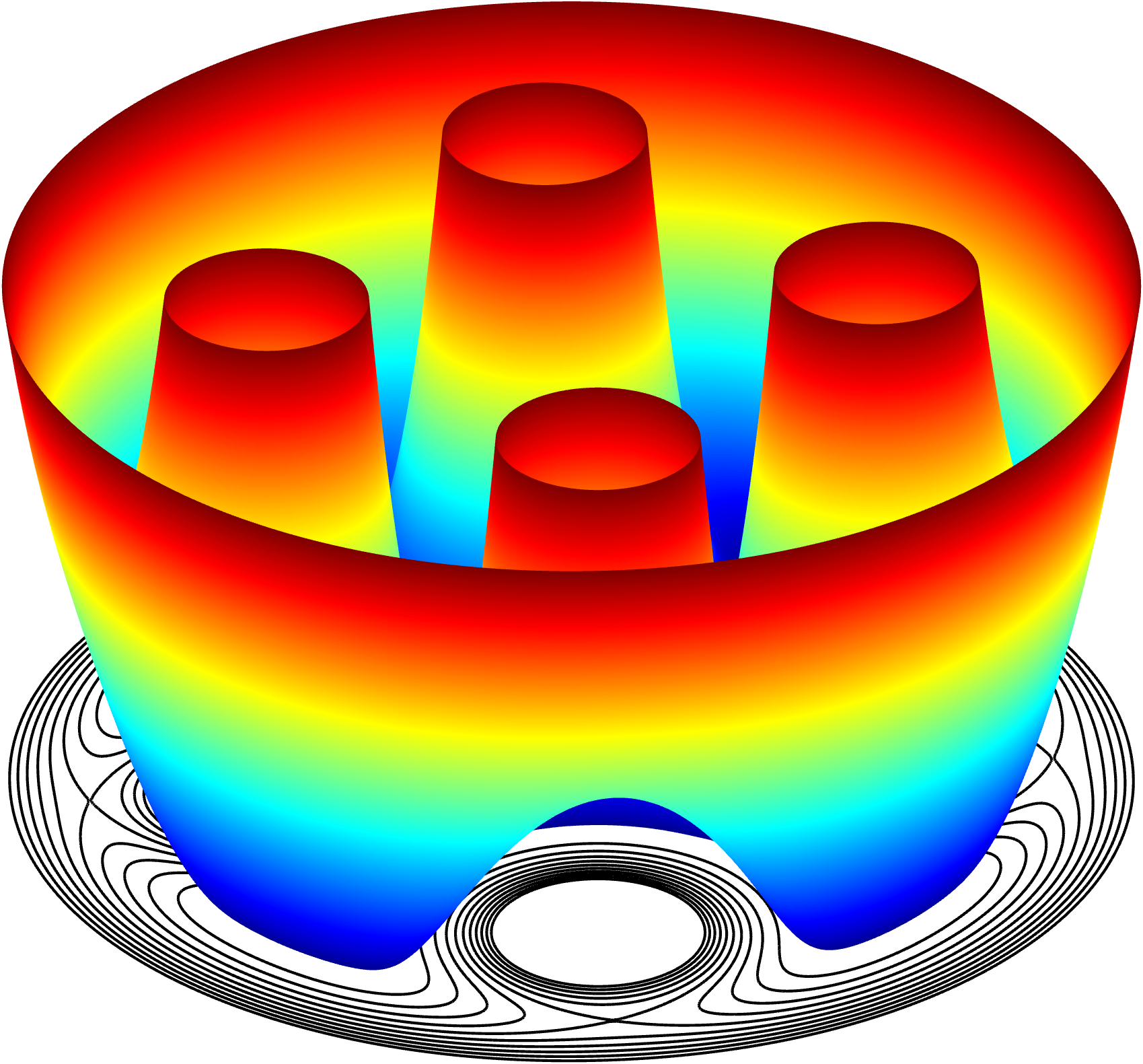}
\vspace{0.02cm}

\small
(b) Numerical solution for $\varphi=7$
\end{minipage}

\vspace{0.45cm} 

\begin{minipage}[t]{0.495\linewidth}
\centering
\includegraphics[width=0.8\linewidth]{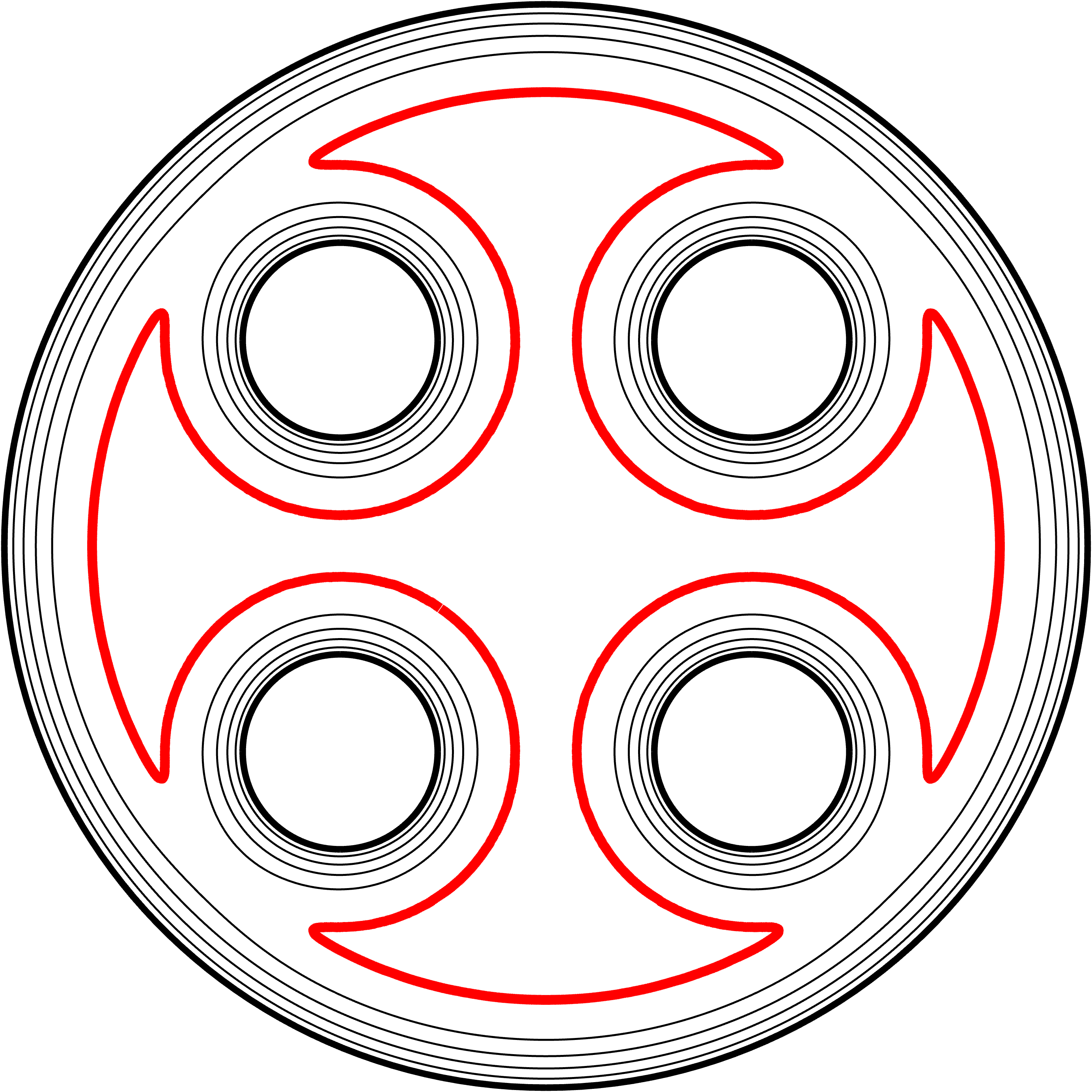}
\vspace{0.02cm}

\small
(c) Connected dead zone, $\varphi=10$
\end{minipage}
\hfill
\begin{minipage}[t]{0.495\linewidth}
\centering
\includegraphics[width=0.85\linewidth]{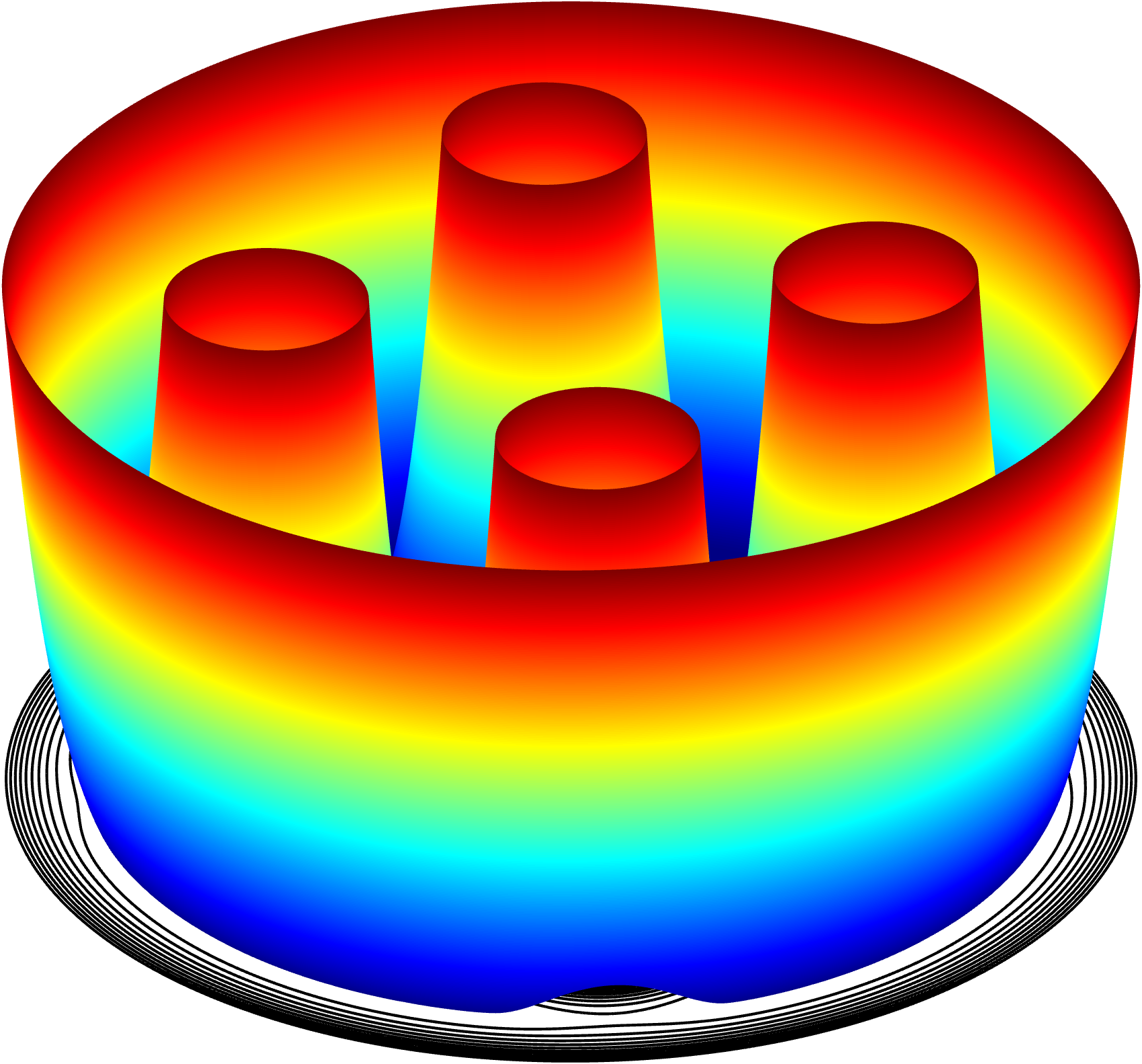}
\vspace{0.02cm}

\small
(d) Numerical solution for $\varphi=10$
\end{minipage}

\vspace{0.45cm}

\begin{minipage}[t]{0.495\linewidth}
\centering
\includegraphics[width=0.8\linewidth]{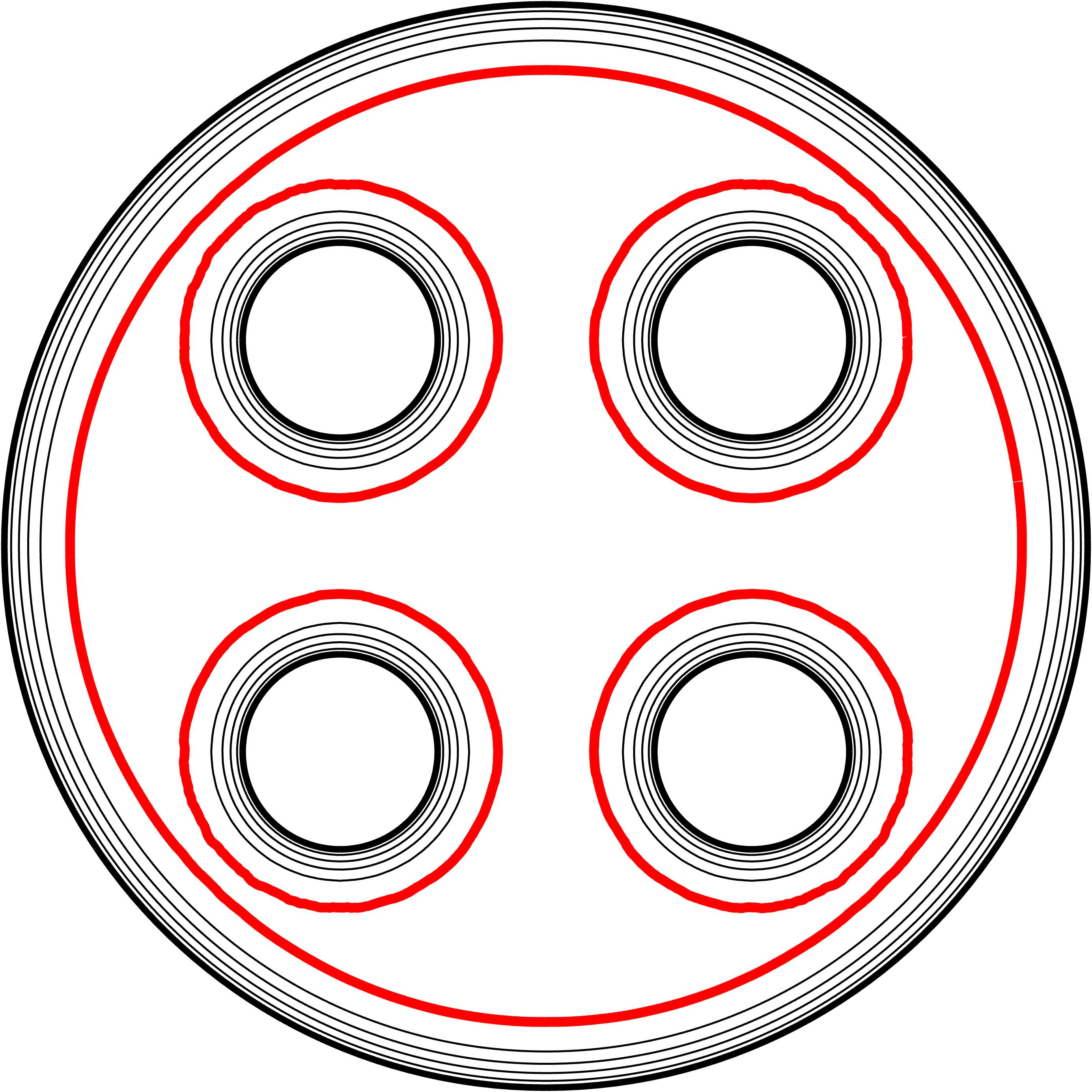}
\vspace{0.02cm}

\small
(e) Connected dead zone, $\varphi=13$
\end{minipage}
\hfill
\begin{minipage}[t]{0.495\linewidth}
\centering
\includegraphics[width=0.85\linewidth]{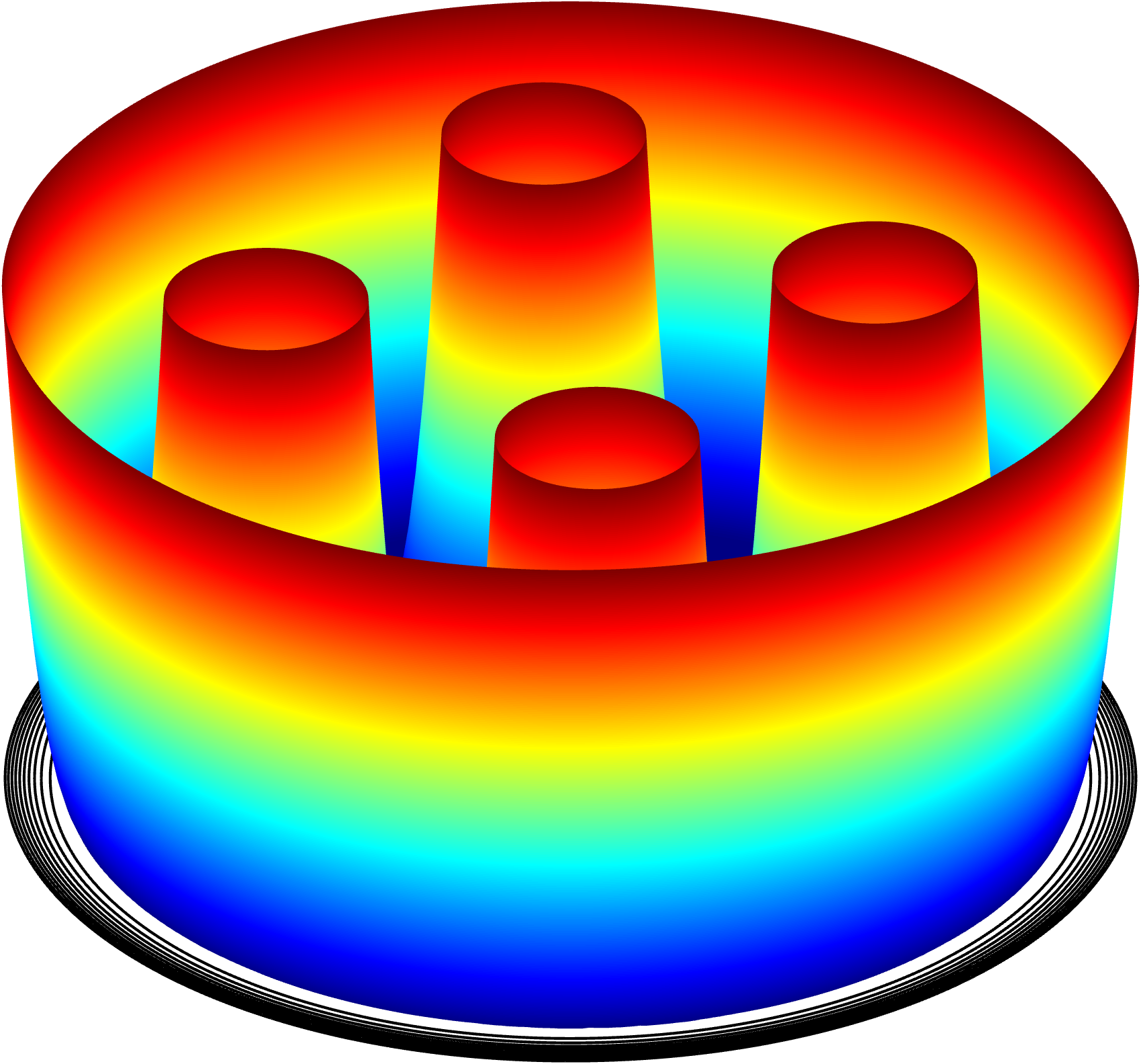}
\vspace{0.02cm}

\small
(f) Numerical solution for $\varphi=13$
\end{minipage}

\caption{
Dead-zone formation in a circular domain with four circular holes for
$p=0.1$. Panels (a)--(b), (c)--(d), and (e)--(f) correspond to
$\varphi=7$, $10$, and $13$, respectively. Red curves indicate the
reconstructed dead-zone interfaces using $\alpha=10^{-4}$.
}
\label{fig:fourholes_deadcore}
\end{figure}

\section{Code structure and availability}\label{sec:code}

The MATLAB implementation is organized into modular components
corresponding to the main stages of the computation. The driver script
\texttt{main.m} initializes the model parameters, constructs
the computational geometry, invokes the nonlinear solver, performs
dead-core detection, computes error measures, and generates graphical
output.

The code is organized into the following directories:

\begin{itemize}
\item \texttt{config/}: parameter definitions and default settings.
\item \texttt{exact/}: reference solutions and validation routines.
\item \texttt{fem/}: geometry construction and finite element assembly.
\item \texttt{solver/}: nonlinear solution algorithms.
\item \texttt{postproc/}: dead-core detection and error evaluation.
\item \texttt{visualization/}: plotting and figure generation.
\item \texttt{utils/}: auxiliary utility functions.
\item \texttt{results/}: automatically generated figures and numerical output.
\end{itemize}

The finite element models support several geometries, including disks,
annuli, hexagons, and multi-hole domains. For the disk benchmark
problem, a high-accuracy radial reference solution is available and is
used for validation and error estimation. For more general geometries,
where reference solutions are unavailable, the software provides
numerical diagnostics and post-processing tools for solution
assessment. Visualization routines produce two-dimensional contour
plots, three-dimensional surface plots, and convergence diagnostics.
Figures can be exported in PNG or PDF format, with all output stored
automatically in the \texttt{results/} directory.

The overall workflow of the software can be summarized as
\[
\texttt{main}
\rightarrow
\texttt{fem}
\rightarrow
\texttt{solver}
\rightarrow
\texttt{postproc}
\rightarrow
\texttt{visualization},
\]
with the optional use of the \texttt{exact/} module for the unit disk benchmark
problem possessing reference solutions.

The modular design facilitates the incorporation of new geometries,
nonlinear solvers, and post-processing procedures while maintaining a
compact and extensible code base.

The MATLAB\textsuperscript{\tiny\textregistered} codes used to generate
the numerical results and figures are available from the corresponding
author upon reasonable request.

\section{Conclusions}\label{sec:conclusions}

This paper investigated dead-zone formation in two-dimensional
diffusion--reaction systems with power-law kinetics and fractional
reaction orders. An IMEX time discretization combined with a lumped
finite element method was employed to compute stationary solutions and
their associated inactive regions. The numerical results demonstrate
that dead-zone formation and topology depend on the combined effects of
the reaction exponent, the Thiele modulus, and the reactor geometry.
Several disconnected dead-zone islands were observed, with further
parameter changes leading to their expansion and, in some cases,
merging into connected regions.

The unit-disk benchmark demonstrated convergence of the finite element
solution against a high-accuracy radial reference solution. The
interface reconstruction study showed that the detected dead-core
boundary depends on both the post-processing threshold and the spatial
resolution. For geometries without an independent reference solution,
the numerical results showed convergence of the discrete energy and
systematic dependence of the reconstructed interfaces on the detection
threshold.

The results for hexagonal and multi-hole geometries demonstrate that
reactor geometry strongly influences the number, shape, and
connectivity of dead zones. The proposed methodology therefore provides
a flexible framework for studying dead-zone formation in general
two-dimensional reactor domains. Future work will focus on systematic
investigation of the parameter regimes governing dead-zone formation
and topology, as well as extensions to non-isothermal and
multicomponent reaction systems and improved interface reconstruction
for complex geometries.

\section*{Acknowledgements}

J.~V. acknowledges institutional support from the Czech Academy
of Sciences associated with the DSc.\ degree award during a visit
to Nazarbayev University, Astana, Kazakhstan, in May 2026. This work
was also financially supported by the Czech Science Foundation
(Project No.~24-10366S), which contributed to the development of
computational methods relevant to the present study.

%This research was supported by research grants 02120FD0351 and 021220FD4851 from Nazarbayev University, 
%and Ulam NAWA grant BPN/ULM/2022/1/00164/DEC/1.  
%\textcolor{red}{Their requirements regarding the acknowledgments:}\\
%Acknowledgments of people, grants, funds, etc. should be placed in a separate section before the
%reference list. The names of funding organizations should be written in full.

% ------------------------------------------------------------------------
%Included for Gather Purpose only:
%input "Xbib.bib"

\bigskip
\end{document}